\documentclass[11pt,a4paper,reqno]{amsart}
\usepackage{amsfonts}\usepackage{amsmath,amssymb,amsthm,amsxtra}
\usepackage{float}
\usepackage{amssymb}
\usepackage{mathabx}
\usepackage{bbm}
\usepackage[colorlinks, linkcolor=blue,anchorcolor=Periwinkle,
citecolor=Red,urlcolor=Emerald]{hyperref}
\usepackage[usenames,dvipsnames]{xcolor}
\usepackage{enumitem}
\usepackage{geometry,array} 
\usepackage{graphicx}
\usepackage{subfigure}
\usepackage{bookmark}
\usepackage{tikz}\usetikzlibrary{matrix}
\usepackage{url}
\usepackage{dsfont}
\usepackage[colorinlistoftodos]{todonotes}
\usepackage{tablists} \restorelistitem
\usepackage{color}

\usetikzlibrary{decorations.markings}
\tikzset{->-/.style={decoration={  markings,  mark=at position #1 with
    {\arrow{>}}},postaction={decorate}}}
\tikzset{-<-/.style={decoration={  markings,  mark=at position #1 with
    {\arrow{<}}},postaction={decorate}}}
\usepackage{extarrows}
\usepackage[all]{xy}
\usepackage{setspace}
\usepackage{thmtools}
\usepackage{thm-restate}
\usepackage{hyperref}
\usepackage{cleveref}
\usepackage{pifont}

\usepackage{booktabs}
\usepackage{tabularx}

\usepackage{mathtools}
\newcommand{\mfu}{\mathbf{u}}
\newcommand{\mfv}{\mathbf{v}}

\newcommand{\mfx}{\mathbf{x}}

\newcommand{\mcA}{\mathcal{A}}

\newcommand{\mcF}{\mathcal{F}}

\newcommand{\mbN}{\mathbb{N}}

\newcommand{\mbQ}{\mathbb{Q}}
\newcommand{\mbR}{\mathbb{R}}

\newcommand{\mbT}{\mathbb{T}}

\newcommand{\mbZ}{\mathbb{Z}}
\theoremstyle{plain}
\newtheorem{theorem}{Theorem}[section]

\newtheorem{lemma}[theorem]{Lemma}

\newtheorem{proposition}[theorem]{Proposition}
\newtheorem{conjecture}[theorem]{Conjecture}
\theoremstyle{definition}
\newtheorem{definition}[theorem]{Definition}

\newtheorem{example}[theorem]{Example}

\newtheorem{remark}[theorem]{Remark}

\numberwithin{equation}{section}
\newtheorem{definition-proposition}[theorem]{Definition-Proposition}

\begin{document}

\title[Log-concavity and unimodality of cluster monomials of type $A_3$]{Log-concavity and unimodality of cluster \\monomials of type $A_3$}

\date{\today}
\author{Zhichao Chen}
\address{School of Mathematical Sciences\\ University of Science and Technology of China \\ Hefei, Anhui 230026, P. R. China}
\email{czc98@mail.ustc.edu.cn}

\maketitle

\begin{abstract}
The log-concavity of cluster variables of type $A_n$ and cluster monomials of type $A_2$ was established by Chen-Huang-Sun \cite{CHS26}. It is still a conjecture for the cluster monomials of higher rank. In this paper, we prove the log-concavity and unimodality of the cluster monomials of type $A_3$, a substantially more intricate case. Moreover, we refine and extend this conjecture by considering the unimodality and the strongly isomorphism of cluster algebras. 
\end{abstract}
%\footnotemark{} 
\renewcommand{\thefootnote}{} 
\footnotetext{\emph{\hspace*{1.3em}Keywords}: log-concavity, unimodality, cluster monomials, type $A_3$.\\ \hspace*{1.3em}\emph{2020 Mathematics Subject Classification}: 13F60, 05E10, 05A20.} 
\renewcommand{\thefootnote}{\arabic{footnote}}
%=======================================
\tableofcontents
%=======================================
\section{Introduction}
Cluster algebras are an important class of commutative algebras characterized by multiple sets of generators and mutation relations among them. In \cite{FZ02, FZ03}, they were first introduced to investigate the total positivity of Lie groups and canonical bases of quantum groups. Since then, cluster algebras have been found to be closely connected with a wide range of areas in mathematics, including representation theory, higher Teichm{\"u}ller theory, integrable system, Poisson geometry, number theory, and combinatorics.
	 
Log-concavity is a fundamental and important property of sequences and polynomials in algebras and combinatorics, such as \cite{Sta89, BH20, HMMS22, CHS26}. It has been widely investigated in algebraic combinatoric, probability theory, algebraic geometry, and representation theory. In \cite{Oko03}, Okounkov conjectured the log-concavity of various objects arising in representation theory from a statistical physics perspective, particularly, the structure constants for many interesting basis from representation theory. 

Cluster monomials play a fundamental role in the representation theory of cluster algebras. The linear independence of the cluster monomials was conjectured by Fomin-Zelevinsky \cite{FZ04} and solved by the method of categorification \cite{IKFP13} and scattering diagrams \cite{GHKK18}. Based on the Laurent phenomenon and positivity in \cite{FZ02, FZ03,GHKK18}, we are interested in the log-concavity for the cluster monomials of type $A_n$, which form an atomic theta basis \cite{Man17, GHKK18}.  Based on \cite{FG09, Shen14}, the theta basis can also reflect many geometric properties in a cluster convex hull, especially of type $A_n$. In \cite{CHS26}, the log-concavity of cluster variables, $F$-polynomials of type $A_n$ and cluster monomials of type $A_2$ was first established, by use of triangulations and $T$-paths \cite{FST08, Sch08, ST09}. However, for the cluster monomials of higher rank, it is quite difficult, even for type $A_3$. This difficulty arises from the rapid growth in the number of variables and the resulting complexity. As a consequence, the following conjecture was proposed in \cite{CHS26}.
\begin{conjecture}[\Cref{An conj}]
	The cluster monomials of type $A_n\ (n\geq 3)$ are log-concave.
\end{conjecture}
For this purpose, we aim to solve this conjecture for the cluster monomials of type $A_3$ in this paper, which is a substantially more intricate case. We first study the permutation and compatibility between seeds and cluster mutations (\Cref{lem: permutation} \& \Cref{lem: sign}). In addition, we introduce the novel notions of unimodal Laurent polynomials and cluster monomials, see \Cref{def: unimodal} and \Cref{def: cm}. Then, we prove the unimodality of the same types considered in \cite{CHS26}.
\begin{proposition}[\Cref{prop: A3 unimodal}]
	The cluster variables of type $A_n$ and the cluster monomials of type $A_2$ are unimodal.
\end{proposition}
We also introduce a combinatorial method by using the binomial coefficients and the convolution, see \Cref{def: convo} and \Cref{prop: LU}, which implies that the log-concavity can be kept under the convolution. We have already known that by taking different initial seeds, the clusters of the same type $A_3$ may have different expressions. In fact, they are the same up to the strongly isomorphism defined in \cite{FZ03}. However, to prove the log-concavity and unimodality of all the cluster monomials of type $A_3$, we can reduce to three special types in \Cref{three cases}. 
 
	 \begin{proposition}[\Cref{prop: classification}]	All the cluster monomials of type $A_3$ are log-concave (or unimodal) if and only if the cluster monomials of type $A_3$ with the initial seeds in \Cref{three cases} are log-concave (or unimodal).
\end{proposition}

Furthermore, we introduce the definitions and properties of Gauss hypergeometric functions and Jacobi polynomials, see \Cref{lem: real roots}. In integral system, these classical special functions provide an analytic framework for controlling the location of zeros. It is worth mentioning that the relations between the integral system and cluster algebras have been widely investigated in \cite{IIKKN13a, IIKKN13b, IY16, Nak24} and so on. Here, with the help of Newton inequalities (\Cref{thm: newton}) and the combinatorial method above, we obtain the following main result. 
\begin{theorem}[\Cref{thm: main result}]
	The cluster monomials of type $A_3$ are log-concave and unimodal.
\end{theorem}
As a consequence, we can refine and extend the conjecture in \cite{CHS26} by considering the unimodality and strongly isomorphism of cluster algebras.
\begin{conjecture}[\Cref{general conj}] We conjecture that the following two statements hold.
	\begin{enumerate}
	\item The cluster monomials of type $A_n$ with $n\geq 4$ are log-concave and unimodal.
		\item The strongly isomorphism of cluster algebras keeps the log-concavity and unimodality of cluster monomials of type $A_n$ with $n\geq 4$. Namely, the log-concavity and the unimodality of cluster monomials are independent of the choice of the initial seed.
	\end{enumerate}
\end{conjecture}

The paper is organized as follows. In \Cref{sec: pre}, we recall some basic notions about seeds, cluster algebras and cluster monomials. Then, we introduce the permutation and compatibility among them, see \Cref{lem: permutation} and \Cref{lem: sign}. In \Cref{sec: comb}, we define the log-concavity (\Cref{log-concave2}) and unimodality (\Cref{def: unimodal}) of Laurent polynomials. Based on the Laurent phenomenon, we can consider these properties for the cluster monomials, see \Cref{def: cm}, \Cref{prop: A3 unimodal} and \Cref{An conj}. We also provide some combinatorial methods by using binomial coefficients and the convolution, see \Cref{prop: LU}. In \Cref{sec: cluster monomials}, we give a sufficient condition (\Cref{prop: classification}) to determine the log-concavity and unimodality of cluster monomials of type $A_3$, see also \Cref{three cases}. In \Cref{sec: main}, by use of Gauss hypergeometric functions, Jacobi polynomials and Newton inequalities, we prove the log-concavity and unimodality of all the cluster monomials of type $A_3$ (\Cref{thm: main result}). Moreover, we give a refined and extended conjecture, see \Cref{general conj}.
\section*{Conventions}
\begin{itemize}[leftmargin=1em]\itemsep=0pt
\item We denote by $\text{Mat}_{n\times n}(\mbZ)$ the set of all $n\times n$ integer square matrices. An integer square matrix $B$ is said to be \emph{skew-symmetrizable} if there exists a positive integer diagonal matrix $D$, such that $DB$ is skew-symmetric and $D$ is called the \emph{left skew-symmetrizer} of $B$. 
\item For any $a\in \mbZ$, we denote $[a]_{+}=\max(a,0)$ and then $a=[a]_+-[-a]_+$.
\item For any $n\in \mbN$, its factorial is defined by $n! = 1\times  2\times\cdots \times n$ for $n\ge 1$, with $0! = 1$.
\end{itemize}

\section{Preliminaries}\label{sec: pre}
In this section, based on \cite{FZ02, FZ03, FZ07, Nak23}, we recall some basic notions and properties about cluster algebras, especially the ones without coefficients.
\subsection{Seeds and cluster algebras} Let $n\in \mbN_{+}$ and $\mcF$ be a rational function field of $n$ variables over $\mbQ$. A (labeled) \emph{seed} is a pair $(\mfx,B)$, such that $\mfx=(x_1,\dots,x_n)$ is an $n$-tuple of algebraically independent and generating elements of $\mcF$ and $B=(b_{ij})_{n\times n}$ is a skew-symmetrizable matrix. The $n$-tuple $\mfx$ is called the \emph{cluster}, elements $x_i\in \mfx$ are called \emph{cluster variables} and $B$ is called the \emph{exchange matrix}. For a seed $(\mfx,B)$ and $1\leq k\leq n$, we define $\mu_k(\mfx,B)=(\mfx^{\prime},B^{\prime})$ such that $\mfx^{\prime}=(x_{1}^{\prime},\dots,x_{n}^{\prime})$, where \begin{align}\label{cluster variable mutation}\ 
		x_{i}^{\prime}=\left\{
		\begin{array}{ll}
			x_{k}^{-1}(\prod\limits_{j=1}^{n}x_{j}^{[b_{jk}]_{+}}+\prod\limits_{j=1}^{n}x_{j}^{[-b_{jk}]_{+}}), &   \text{if}\ i=k, \\
			x_{i}, &  \text{if}\ i \neq k, 
		\end{array} \right.
	\end{align}
and $B^{\prime}=(b_{ij}^{\prime})_{n\times n}$ is given by 
	\begin{align} \label{matrix mutation}
		b_{ij}^{\prime}=\left\{
		\begin{array}{ll}
			-b_{ij}, &  \text{if}\ i=k \;\;\mbox{or}\;\; j=k, \\
			b_{ij}+\frac{1}{2}(b_{ik}|b_{kj}|+|b_{ik}|b_{kj}), &  \text{if}\ i\neq k \;\;
			\mbox{and}\; j\neq k. 
		\end{array} \right. 
	\end{align}
In fact, $(\mfx^{\prime},B^{\prime})$ is still a seed and $\mu_k$ is involutive, that is $\mu_k(\mfx^{\prime},B^{\prime})=(\mfx,B)$, see \cite{Nak23}. Then, $(\mfx^{\prime},B^{\prime})$ is called the \emph{k-direction mutation} of $(\mfx,B)$. Two seeds (exchange matrices) are said to be \emph{mutation-equivalent} if one can be obtained from the other by a finite sequence of mutations. The mutation equivalence class of $\Sigma$ is denoted by $\overline{\Sigma}$. A \emph{cluster pattern} $\mathbf{\Sigma}=\{(\mfx_t,B_t)|\ t\in \mbT_n\}$ is a collection of seeds which are labeled by the vertices of $n$-regular tree $\mbT_{n}$, such that $(\mfx_{t^{\prime}},B_{t^{\prime}})=\mu_k(\mfx_t,B_t)$ for any edge $t \stackrel{k}{\longleftrightarrow} t^{\prime}$ in $\mbT_n$. 

In the following, we use the notations that $\mathbf{x}_{t}=(x_{1;t},\dots,x_{n;t})$ and $B_{t}=(b_{ij;t})_{n\times n}$. Note that for an arbitrary fixed vertex $t_0\in \mbT_n$, we call the seed $(\mfx_{t_0},B_{t_0})$ \emph{initial seed} and denote by $\mathbf{x}_{t_0}=\mfx=(x_{1},\dots,x_{n})$ the \emph{initial cluster} and $B_{t_0}=B=(b_{ij})_{n\times n}$ the \emph{initial exchange matrix}.
\begin{definition}[\emph{Cluster algebra}]\label{cluster algebras}
	For a cluster pattern $\mathbf{\Sigma}$, the  \emph{cluster algebra} $\mcA=\mcA(\mathbf{\Sigma})$ is the $\mbQ$-subalgebra of $\mcF$ generated by all the cluster variables $\{x_{i;t}|\ i=1,\dots,n;t\in\mbT_{n}\}$. Here, $n$ is  the \emph{rank} of $\mcA$.\end{definition}
The most important properties of cluster algebras are the Laurent phenomenon and the positivity.
	\begin{theorem}[{\cite[Theorem 3.1]{FZ02}, \cite[Theorem 4.10]{GHKK18}}]
\label{theorem:LE}
In a cluster algebra $\mcA$, any cluster variable is a Laurent polynomial in terms of the initial cluster with nonnegative integer coefficients.
\end{theorem}
More precisely, the Laurent expression of $x_{i;t}$ in terms of the initial cluster $\mfx=(x_1,\dots,x_n)$ can be written as follows:\begin{align}
	x_{i;t}=\dfrac{N_{i;t}(x_1,\dots, x_n)}{x_1^{d_{1i;t}}\cdots x_n^{d_{ni;t}}},\  d_{ji;t}\in \mbZ, \label{express of d-vector}
\end{align} where $N_{i;t}(x_1,\dots, x_n)$ is a polynomial with nonnegative integer coefficients and not divisible by any $x_j$.
	\begin{definition}[\emph{Finite type}]
A cluster algebra $\mcA$ is called of \emph{finite type} if it contains only finite distinct seeds. Otherwise, it is called of \emph{infinite type}.
\end{definition}
It was shown in \cite[Theorem 1.8]{FZ03} that a cluster algebra is of finite type if and only if it contains an exchange matrix $B$ whose Cartan counterpart $A(B)$ is a Cartan matrix of finite type, that is of type $A_n\ (n\geq 1),B_n\ (n\geq 2),C_n\ (n\geq 2),D_n\ (n\geq 4),E_{6,7,8},F_4,G_2$. For our purpose, throughout this paper, we focus on the cluster algebras of type $A_n$, especially about the type $A_3$ in \Cref{dynkin A3}.
\\

\begin{figure}[htpb]

\tikzset{every picture/.style={line width=0.75pt}} %set default line width to 0.75pt        

\begin{tikzpicture}[x=0.75pt,y=0.75pt,yscale=-1,xscale=1,scale=0.75]

%Straight Lines [id:da5107648810420541] 
\draw    (100,123) -- (227,123) ;
\draw [shift={(227,123)}, rotate = 0] [color={rgb, 255:red, 0; green, 0; blue, 0 }  ][fill={rgb, 255:red, 0; green, 0; blue, 0 }  ][line width=0.75]      (0, 0) circle [x radius= 3.35, y radius= 3.35]   ;
\draw [shift={(100,123)}, rotate = 0] [color={rgb, 255:red, 0; green, 0; blue, 0 }  ][fill={rgb, 255:red, 0; green, 0; blue, 0 }  ][line width=0.75]      (0, 0) circle [x radius= 3.35, y radius= 3.35]   ;
%Straight Lines [id:da41568138217616346] 
\draw    (227,123) -- (354,123) ;
\draw [shift={(354,123)}, rotate = 0] [color={rgb, 255:red, 0; green, 0; blue, 0 }  ][fill={rgb, 255:red, 0; green, 0; blue, 0 }  ][line width=0.75]      (0, 0) circle [x radius= 3.35, y radius= 3.35]   ;

\end{tikzpicture}
\caption{Dynkin diagram of type $A_3$}
\label{dynkin A3}
\end{figure}
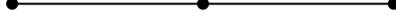
\begin{definition}[\emph{Cluster monomial}]\label{cluster monomials}
	A \emph{cluster monomial} is a product of nonnegative powers of cluster variables which all belong to a same cluster. More precisely, for any $t\in \mbT_n$, the set of cluster monomials is $\{x_{1;t}^{m_1}\dots x_{n;t}^{m_n}\mid m_1,\dots,m_n\in \mbZ_{\geq 0}\}$.
\end{definition}
Note that all the cluster variables belong to the set of all the cluster monomials.
\begin{remark}
	Cluster monomials play a fundamental role in  cluster theory. In \cite[Theorem 7.20]{GHKK18}, it was proved that for a cluster algebra, all distinct cluster monomials are linearly independent over $\mbZ$. In addition, the cluster monomials are contained in the \emph{theta functions}. When the cluster algebra is of finite type, theta functions are strictly cluster monomials, also see \cite{Nak23}. 
\end{remark}
\subsection{Permutation and compatibility} In this subsection, we introduce a permutation action on the seeds and its compatibility with the cluster mutations.
\begin{definition}[\emph{Permutation action}]\label{def: per}
	For a seed $\Sigma=(\mfx,B)$ in $\mcF$ and a permutation $\sigma \in \mathfrak{S}_n$, we define the action of $\sigma$ on $\Sigma$ by 
	\begin{align}
		\sigma(\Sigma)=(\sigma(\mfx),\sigma(B)),
	\end{align} where $\sigma(\mfx)=\mfx^{\prime}$ and $\sigma(B)=B^{\prime}$ are defined by 
	\begin{align}
		x^{\prime}_i=x_{\sigma^{-1}(i)},\ b^{\prime}_{ij}=b_{\sigma^{-1}(i)\sigma^{-1}(j)}.
	\end{align}
\end{definition} It is clear that $\sigma(\Sigma)$ is still a seed in $\mcF$, which yields a left action of $\mathfrak{S}_n$ on the set of seeds in $\mcF$. That is to say, we have $\sigma_2(\sigma_1(\Sigma))=(\sigma_2\sigma_1)(\Sigma)$ for any $\sigma_1,\sigma_2 \in \mathfrak{S}_n$. Sometimes, we also call $\Sigma$ and $\sigma(\Sigma)$ the same \emph{unlabeled seed}. Let the permutation matrix $P_\sigma$ associated with $\sigma\in \mathfrak{S}_n$ be as follows:
\begin{align}
	P_\sigma=(p_{ij}),\ p_{ij}=\delta_{i,\sigma^{-1}(j)}.
\end{align} Then, we can get the lemma as follows by direct calculation.
\begin{lemma}[{\cite{Nak23}}]\label{lem: permutation}
	For a seed $\Sigma=(\mfx,B)$ in $\mcF$ and $\sigma \in \mathfrak{S}_n$, the following equalities hold:
	\begin{enumerate}
		\item $\sigma(B)=P_{\sigma}^{\top}BP_{\sigma}$.
		\item $\mu_{\sigma(k)}(\sigma(\Sigma))=\sigma(\mu_k(\Sigma))$.
	\end{enumerate}
\end{lemma}
\begin{lemma}\label{lem: sign}
	For any two seeds $\Sigma_1=(\mfx,B)$ and $\Sigma_2=(\mfx,-B)$ in $\mcF$ with the same initial cluster $\mfx$, their mutation equivalence classes $\overline{\Sigma}_1$ and $\overline{\Sigma}_2$ are the same up to the sign of exchange matrices.
\end{lemma}
\begin{proof}
	According to the mutation rules \eqref{matrix mutation} of the exchange matrices, we may obtain the same mutated exchange matrices up to a sign under the same mutation sequence. Namely, for any $r\in \mbN$ and $1\leq i_j\leq n$ with $j=1,\dots,r$, we have  
	\begin{align}
		\mu_{i_r}\mu_{i_{r-1}}\dots\mu_{i_1}(B)=-\mu_{i_r}\mu_{i_{r-1}}\dots\mu_{i_1}(-B).
	\end{align} Moreover, since the mutation functions \eqref{cluster variable mutation} of cluster variables are binomial and symmetric with different signs of the exchange matrices, we get the same cluster variables. That is to say,  
	\begin{align}
		\mu_{i_r}\mu_{i_{r-1}}\dots\mu_{i_1}(\mfx(B))=\mu_{i_r}\mu_{i_{r-1}}\dots\mu_{i_1}(\mfx(-B)),
	\end{align} where $\mfx(B)$ and $\mfx(-B)$ denotes the same initial cluster $\mfx$ but with different initial exchange matrices. Then, we complete the proof.
\end{proof}
\section{A combinatorial framework}\label{sec: comb} In this section, based on \cite{CHS26}, we recall the definitions and properties of the log-concavity of Laurent polynomials and cluster monomials. Here, we also introduce a new notion of unimodal Laurent polynomials.
\subsection{Log-concave and unimodal Laurent polynomials} We begin with recalling the notion of log-concavity of Laurent polynomials defined in \cite{CHS26}. \begin{definition}[\emph{Log-concavity of Laurent polynomials}]\label{log-concave2}
	A nonzero Laurent polynomial with nonnegative real coefficients and $m$ variables \begin{align}f(x_1,\dots,x_m)=\sum\limits_{i_1=l_1}^{n_1}\dots\sum\limits_{i_m=l_m}^{n_m}a_{i_1,\dots,i_m}x_1^{i_1}\dots x_m^{i_m} \label{Laurent log} 
	\end{align} is said to be \emph{log-concave} if for any $1\leq j\leq m$ and $l_j\leq i_j\leq n_j$, \begin{align}a^2_{i_1,\dots,i_j,\dots,i_m}\geq a_{i_1,\dots,i_{j}-1,\dots,i_m}a_{i_1,\dots,i_{j}+1,\dots,i_m},\label{Laurent inequality}\end{align} where setting $a_{i_1,\dots,l_j-1,\dots,i_m}=a_{i_1,\dots,n_{j}+1,\dots,i_m}=0$. 
\end{definition}
We say that $f(x_1,\dots,x_n)$ has \emph{no internal zero coefficients} if for any $1\leq j\leq m$, there does not exist integers  $l_j\leq p_j<q_j<r_j\leq n_j$, satisfying $a_{i_1,\dots,p_j,\dots,i_m}\neq 0$, $a_{i_1,\dots,q_j,\dots,i_m}= 0$ and $a_{i_1,\dots,r_j,\dots,i_m}\neq 0$. Then, we introduce a novel notion of unimodality for  Laurent polynomials. 

\begin{definition}[\emph{Unimodality of Laurent polynomials}] \label{def: unimodal}
	Let the nonzero Laurent polynomial $f(x_1,\dots,x_n)$ with positive real coefficients be the same as \eqref{Laurent log}. It is said to be \emph{unimodal} if for any $1\leq j\leq m$, there exists $l_j\leq k_j\leq n_j$, such that \begin{align}\label{eq: unimodal} a_{i_1,\dots,l_{j},\dots,i_m}\leq \cdots \leq a_{i_1,\dots,k_{j},\dots,i_m}\geq  \cdots \geq  a_{i_1,\dots,n_{j},\dots,i_m} \end{align} holds for any $i_1,\dots, i_{j-1}, i_{j+1}, \dots i_m$.  
\end{definition}
For convenience, we sometimes say that the sequence of the coefficient sequence of $f(x_1,\dots,x_n)$ is log-concave or unimodal if $f(x_1,\dots,x_n)$ is log-concave or unimodal.
\begin{lemma}\label{lem: log is uni}
	A nonnegative log-concave Laurent polynomial $f(x_1,\dots,x_m)$ with no internal zeros is unimodal. 
\end{lemma}
\begin{proof}
	 For any $j\in \{1,\dots,m\}$, we fix the indices $I_j=\{i_1,\dots,i_{j-1},i_{j+1},\dots,i_m\}$ except $i_j$. Note that $f$ has no internal zeros. If there are fewer than two consecutive terms associated with $I_j$ having positive coefficients, then \eqref{eq: unimodal} holds trivially. Hence, we may assume that there at least three consecutive terms associated with $I_j$ with positive coefficients. Denote by 
	\begin{align}
		r_{t;I_j}=\dfrac{a_{i_1,\dots,t+1,\dots,i_m}}{a_{i_1,\dots,t,\dots,i_m}},\ (a_{i_1,\dots,t,\dots,i_m}>0).
	\end{align} Since $f(x_1,\dots,x_n)$ is log-concave, we have 
	\begin{align}
		\dfrac{a_{i_1,\dots,t+1,\dots,i_m}}{a_{i_1,\dots,t,\dots,i_m}}\leq \dfrac{a_{i_1,\dots,t,\dots,i_m}}{a_{i_1,\dots,t-1,\dots,i_m}}.
	\end{align} Hence, it implies that $r_{t;I_j} \leq r_{t-1;I_j}$, that is $\{r_{t;I_j}\}_{t=p_j}^{q_j}$ is a monotonically decreasing sequence, where $l_j\leq p_j< q_j\leq n_j$. If the first term $r_{p_j;I_j}$ of this sequence is no larger than $1$, then the sequence $\{r_{t;I_j}\}$ admits $1$ as an upper bound. Thus, we have 
	\begin{align}
		a_{i_1,\dots,p_j,\dots,i_m}\geq \cdots \geq  a_{i_1,\dots,t-1,\dots,i_m} \geq a_{i_1,\dots,t,\dots,i_m} \geq  a_{i_1,\dots,t+1,\dots,i_m}\geq \cdots \geq  a_{i_1,\dots,q_j,\dots,i_m}.
	\end{align} Note that the remaining coefficients are zero, we may directly obtain \eqref{eq: unimodal} by taking $k_j=p_j$. If the last term $r_{q_j;I_j}$ of this sequence is no less than $1$, by a similar argument, we have \eqref{eq: unimodal} by taking $k_j=q_j$. If $r_{s;I_j}\geq 1$ and $r_{s+1;I_j}< 1$ for some $p_j< s< q_j$, we have 
	\begin{align}
		a_{i_1,\dots,p_j,\dots,i_m}\leq \cdots \leq  a_{i_1,\dots,s,\dots,i_m} \leq a_{i_1,\dots,s+1,\dots,i_m} \geq  a_{i_1,\dots,s+2,\dots,i_m}\geq \cdots \geq  a_{i_1,\dots,q_j,\dots,i_m}.
	\end{align} Then, we obtain \eqref{eq: unimodal} by taking $k_j=s+1$, which completes the proof.
\end{proof} 
\begin{example}
	The Laurent polynomial with three variables \begin{align}f(x_1,x_2,x_3)=\dfrac{1+2x_2+x_2^2+x_1x_3}{x_1x_2x_3}\end{align} is log-concave by direct calculation. Moreover, it is unimodal since it only has nonnegative coefficients with no internal zeros. In fact, it is a cluster variable of type $A_3$, see \Cref{table inward} in the next \Cref{sec: cluster monomials}.
\end{example}
Thanks to the Laurent phenomenon and the positivity  of cluster variables (\Cref{theorem:LE}), cluster monomials share these properties,  which ensures that the following definition is well defined.
\begin{definition}[\emph{Log-concavity and unimodality of cluster monomials}]\label{def: cm}
	For a coefficient-free cluster algebra $\mcA$, a cluster monomial is called \emph{log-concave} (\emph{unimodal}) if it is log-concave (unimodal) as a Laurent polynomial with respect to the initial cluster $\mfx$.
\end{definition}
\begin{remark}\label{same up}
	Note that any two Laurent polynomials which are same up to a Laurent monomial factor keep the same property of log-concavity and unimodality. In particular,  by  formula \eqref{express of d-vector}, any cluster variable $x_{i;t}$ is log-concave (unimodal) if and only if $N_{i;t}(x_1,\dots,x_n)$ is log-concave (unimodal).
\end{remark}
In the previous work \cite{CHS26}, we have proved the following theorem and given a log-concavity conjecture for the cluster monomials of type $A_n$ ($n\geq 3$). 
\begin{theorem}[{\cite[Theorem 4.2 \& Theorem 5.3]{CHS26}}]\label{thm: log-concave of cluster variable}
	The cluster variables of type $A_n$ and the cluster monomials of type $A_2$ are log-concave.
\end{theorem}

Moreover, based on the unimodality defined as above, we obtain the following proposition.
\begin{proposition}\label{prop: A3 unimodal}
	The cluster variables of type $A_n$ and the cluster monomials of type $A_2$ are unimodal.
\end{proposition}
\begin{proof}
	In the proof of \cite[Theorem 4.2 \& Theorem 5.3]{CHS26}, by the direct calculation, we obtain that their coefficients are nonnegative and have no internal zeros. Hence, by \Cref{lem: log is uni} and \Cref{thm: log-concave of cluster variable}, we conclude that the cluster variables of type $A_n$ and the cluster monomials of type $A_2$ are unimodal.
\end{proof}
\begin{remark}\label{not log}
	Note that the product of two log-concave Laurent polynomials is not necessarily log-concave. For example, let us consider  \begin{align}f(x_1,x_2)&=\dfrac{2 + 3 x_1 + x_1^2 + 3 x_2 + 4 x_1 x_2 + 3 x_1^2 x_2 + 4 x_1 x_2^2}{x_1x_2},\\ 
	g(x_1,x_2)&=\dfrac{5 + 5 x_1 + x_1^2 + 4 x_2 + 4 x_1 x_2 + x_1^2 x_2 + x_1^2 x_2^2}{x_1x_2}.\end{align} Both of them are log-concave. However, by direct calculation, in the numerator of $f(x_1,x_2)g(x_1,x_2)$, there are three consecutive terms $23x^ 3_1x_2^2,8x^3_1x_2^3$ and $4x^3_1x_2^4$ such that $8^2<23\times 4$. It implies that $f(x_1,x_2)g(x_1,x_2)$ is not log-concave. Hence, it seems not enough to directly get the log-concavity of cluster monomials of type $A_n$ by \Cref{thm: log-concave of cluster variable}. 
\end{remark} 
\begin{conjecture}[{\cite[Conjecture 5.5]{CHS26}}]
	\label{An conj}
	The cluster monomials of type $A_n\ (n\geq 3)$ are log-concave.
\end{conjecture}
\subsection{Binomial coefficients and the convolution} In this subsection, we introduce some properties of the binomial coefficients and the convolution of two log-concave sequences, which can also be referred to \cite{Sta89}.
\begin{lemma}\label{lem: binomial1}
	For $1\leq k\leq n-1$, the inequality related with binomial coefficients holds: 
	\begin{align}
		\binom{n}{k}^2 \geq \binom{n}{k-1}\binom{n}{k+1}.
	\end{align}
\end{lemma}
\begin{proof} Note that the following inequality 
	\begin{align}
			\frac{\binom{n}{k}^2}{\binom{n}{k-1}\binom{n}{k+1}}=\frac{(n-k+1)(k+1)}{(n-k)k}\geq 1.
		\end{align}
\end{proof}
\begin{lemma}\label{lem: binomial2}
	For $0\leq k\leq n-1$, the inequality related with binomial coefficients holds: \begin{align}\binom{n}{k}^2\geq \binom{n-1}{k}\binom{n+1}{k}.\end{align}
\end{lemma}
\begin{proof}
	We only need to note that \begin{align}
		\frac{\binom{n}{k}^2}{\binom{n-1}{k}\binom{n+1}{k}}=\frac{(n-k+1)n}{(n-k)(n+1)}=\frac{n^2-kn+n}{n^2-kn+n-k}\geq 1. 
	\end{align} 
\end{proof}
\begin{definition}[\emph{Convolution sequence}]\label{def: convo}
	For any number sequence $\textbf{a}=\{a_i\}_{i=0}^{n}$ and $\textbf{b}=\{b_i\}_{i=0}^{m}$, we define the \emph{convolution sequence} $\textbf{a}*\textbf{b}=\{(\textbf{a}*\textbf{b})_t\}_{t=0}^{m+n}$ as follows:
	\begin{align}
		(\textbf{a}*\textbf{b})_t  \coloneqq \sum_{\substack{0 \le i \le n,\, 0 \le j \le m, \\ i+j=t} } a_ib_j.
	\end{align}
\end{definition}
\begin{lemma}\label{lem: extended}
	Let $\textbf{a}=\{a_i\}_{i=0}^{n}$ be a nonnegative log-concave sequence with no internal zeros. Then, for any $0\leq i\leq j\leq n$ and $r\geq 0$, we have $a_ia_j\geq a_{i-r}a_{j+r}$.
\end{lemma}
\begin{proof}
	Without loss of generality, we might assume that $a_{i-r}a_{j+r}>0$. Since the sequence $\bf{a}$ has no internal zeros, we have $a_k>0$ for any $i-r \leq k \leq i+r$. By the log-concavity of $\bf{a}$, we have \begin{align}
		\dfrac{a_k}{a_{k-1}}\geq \dfrac{a_{k+1}}{a_{k}},
	\end{align} where $a_{k-1}a_k >0$. This also implies that the ratio sequence of $\bf{a}$ is a monotonically decreasing sequence.
	 Note that 
	\begin{align}
		\dfrac{a_i}{a_{i-r}}=\prod^{r-1}_{k=0} \dfrac{a_{i-k}}{a_{i-k-1}} \geq \prod^{r-1}_{k=0} \dfrac{a_{j+k+1}}{a_{j+k}}= \dfrac{a_{j+r}}{a_{j}}.
	\end{align} Hence, we obtain the inequality $a_ia_j\geq a_{i-r}a_{j+r}$.
\end{proof}
\begin{proposition}\label{prop: LU}
	 Let $\textbf{a}=\{a_i\}_{i=0}^{n}$ and $\textbf{b}=\{b_i\}_{i=0}^{m}$ be two nonnegative log-concave sequences with no internal zeros. Then, the convolution sequence $\textbf{a}*\textbf{b}$ is log-concave and unimodal. 
\end{proposition}
\begin{proof}
	Since $\bf{a}$ and $\bf{b}$ have no internal zeros, it also holds for their convolution $\textbf{a}*\textbf{b}$. We define two $(m+n+1)\times (m+n+1)$ matrices as follows:
	\begin{align}
		A=(a_{ij}),\  a_{ij}=\left\{
		\begin{array}{ll}
			a_{j-i} &  \text{if}\ i\leq j \leq i+n, \\
			0 &  \text{if}\ i>j \;\;
			\mbox{or}\; j> i+n. 
		\end{array} \right. \\
		B=(b_{ij}),\  b_{ij}=\left\{
		\begin{array}{ll}
			b_{j-i} &  \text{if}\ i\leq j \leq i+m, \\
			0 &  \text{if}\ i>j \;\;
			\mbox{or}\; j> i+m. 
		\end{array} \right.
	\end{align} More precisely, we can express $A$ (similar for $B$) as 
	\begin{align}
		A= \scalebox{0.85}{$
\begin{pmatrix}
a_0 & a_1 & a_2 & \cdots & a_n & 0 & \cdots & 0 & 0\\
0      & a_0 & a_1 & \cdots & a_{n-1} & a_n & \cdots & 0  & 0\\
0      & 0      & a_0 & \cdots & a_{n-2} & a_{n-1} & \cdots & 0 & 0 \\
\vdots & \vdots & \vdots &  & \vdots & \vdots & & \vdots & \vdots \\
0      & 0      & 0      & \cdots & a_0 & a_1 & \cdots & a_{m-1} & a_m\\ 0& 0 & 0 &\cdots & 0 & a_0 & \cdots & a_{m-2} & a_{m-1}\\ \vdots & \vdots & \vdots & & \vdots & \vdots & & \vdots & \vdots \\ 0 & 0 & 0 & \cdots & 0 & 0 & \cdots & a_0 & a_1 \\ 0 & 0 & 0 & \cdots & 0 & 0 & \cdots & 0 & a_0
\end{pmatrix}$}.
	\end{align} According to \Cref{lem: extended}, all the $2\times 2$ minors of $A$ and $B$ are nonnegative. By the Cauchy-Binet theorem, this property also holds for $AB$. Note that the entries of $AB$ are exactly the corresponding numbers in $\textbf{a}*\textbf{b}$. Hence, by \Cref{lem: log is uni}, we conclude that $\textbf{a}*\textbf{b}$ is log-concave and unimodal. 
\end{proof}
\section{Classification of the cluster monomials of type $A_3$}\label{sec: cluster monomials}
In this section, we give a classification to determine the log-concavity and unimodality of the cluster monomials of type $A_3$, which essentially consists of three distinct types.
\begin{proposition}\label{prop: classification}
	All the cluster monomials of type $A_3$ are log-concave (or unimodal) if and only if the cluster monomials of type $A_3$ with the initial seeds in \Cref{three cases} are log-concave (or unimodal).
\end{proposition}
\begin{proof}
	Note that all the non-isomorphic quivers of type $A_3$ are listed in \Cref{Non-isomorphic quivers of type A3}, up to the permutation of the vertices. Hence, by \Cref{lem: permutation}, the log-concavity (or unimodality) of cluster monomials up to a permutation of the initial cluster variables is the same. Moreover, note that the exchange matrices of the case $(1)$ and $(2)$ are the same up to a sign. Then, by \Cref{lem: sign}, all the clusters and cluster variables are the same. Thus, it is sufficient to consider the log-concavity (or unimodality) associated with the three initial quivers in \Cref{three cases}, which we call the \emph{three reduced cases}.
\end{proof}
\begin{figure}[htpb]
\tikzset{every picture/.style={line width=0.75pt}} %set default line width to 0.75pt        

\begin{tikzpicture}[x=0.75pt,y=0.75pt,yscale=-1,xscale=1, scale=0.8]
%uncomment if require: \path (0,506); %set diagram left start at 0, and has height of 506

%Straight Lines [id:da8798451956862439] 
\draw    (65.67,124) -- (154.17,123.5) ;
\draw [shift={(154.17,123.5)}, rotate = 359.68] [color={rgb, 255:red, 0; green, 0; blue, 0 }  ][fill={rgb, 255:red, 0; green, 0; blue, 0 }  ][line width=0.75]      (0, 0) circle [x radius= 3.35, y radius= 3.35]   ;
\draw [shift={(65.67,124)}, rotate = 359.68] [color={rgb, 255:red, 0; green, 0; blue, 0 }  ][fill={rgb, 255:red, 0; green, 0; blue, 0 }  ][line width=0.75]      (0, 0) circle [x radius= 3.35, y radius= 3.35]   ;
%Straight Lines [id:da05319768336134767] 
\draw    (65.67,124) -- (152.17,123.51) ;
\draw [shift={(154.17,123.5)}, rotate = 179.68] [color={rgb, 255:red, 0; green, 0; blue, 0 }  ][line width=0.75]    (10.93,-3.29) .. controls (6.95,-1.4) and (3.31,-0.3) .. (0,0) .. controls (3.31,0.3) and (6.95,1.4) .. (10.93,3.29)   ;
%Straight Lines [id:da24179526539062446] 
\draw    (154.17,123.5) -- (242.67,123) ;
\draw [shift={(242.67,123)}, rotate = 359.68] [color={rgb, 255:red, 0; green, 0; blue, 0 }  ][fill={rgb, 255:red, 0; green, 0; blue, 0 }  ][line width=0.75]      (0, 0) circle [x radius= 3.35, y radius= 3.35]   ;
%Straight Lines [id:da6015779573809611] 
\draw    (156.17,123.49) -- (198.01,123.25) -- (242.67,123) ;
\draw [shift={(242.67,123)}, rotate = 359.68] [color={rgb, 255:red, 0; green, 0; blue, 0 }  ][fill={rgb, 255:red, 0; green, 0; blue, 0 }  ][line width=0.75]      (0, 0) circle [x radius= 3.35, y radius= 3.35]   ;
\draw [shift={(154.17,123.5)}, rotate = 359.68] [color={rgb, 255:red, 0; green, 0; blue, 0 }  ][line width=0.75]    (10.93,-3.29) .. controls (6.95,-1.4) and (3.31,-0.3) .. (0,0) .. controls (3.31,0.3) and (6.95,1.4) .. (10.93,3.29)   ;
%Straight Lines [id:da26012955859325027] 
\draw    (306.33,123) -- (394.83,122.5) ;
\draw [shift={(394.83,122.5)}, rotate = 359.68] [color={rgb, 255:red, 0; green, 0; blue, 0 }  ][fill={rgb, 255:red, 0; green, 0; blue, 0 }  ][line width=0.75]      (0, 0) circle [x radius= 3.35, y radius= 3.35]   ;
\draw [shift={(306.33,123)}, rotate = 359.68] [color={rgb, 255:red, 0; green, 0; blue, 0 }  ][fill={rgb, 255:red, 0; green, 0; blue, 0 }  ][line width=0.75]      (0, 0) circle [x radius= 3.35, y radius= 3.35]   ;
%Straight Lines [id:da0548012063696135] 
\draw    (308.33,122.99) -- (394.83,122.5) ;
\draw [shift={(394.83,122.5)}, rotate = 359.68] [color={rgb, 255:red, 0; green, 0; blue, 0 }  ][fill={rgb, 255:red, 0; green, 0; blue, 0 }  ][line width=0.75]      (0, 0) circle [x radius= 3.35, y radius= 3.35]   ;
\draw [shift={(306.33,123)}, rotate = 359.68] [color={rgb, 255:red, 0; green, 0; blue, 0 }  ][line width=0.75]    (10.93,-3.29) .. controls (6.95,-1.4) and (3.31,-0.3) .. (0,0) .. controls (3.31,0.3) and (6.95,1.4) .. (10.93,3.29)   ;
%Straight Lines [id:da8847777188324074] 
\draw    (394.83,122.5) -- (483.33,122) ;
\draw [shift={(483.33,122)}, rotate = 359.68] [color={rgb, 255:red, 0; green, 0; blue, 0 }  ][fill={rgb, 255:red, 0; green, 0; blue, 0 }  ][line width=0.75]      (0, 0) circle [x radius= 3.35, y radius= 3.35]   ;
%Straight Lines [id:da9226699429359215] 
\draw    (394.83,122.5) -- (481.33,122.01) ;
\draw [shift={(483.33,122)}, rotate = 179.68] [color={rgb, 255:red, 0; green, 0; blue, 0 }  ][line width=0.75]    (10.93,-3.29) .. controls (6.95,-1.4) and (3.31,-0.3) .. (0,0) .. controls (3.31,0.3) and (6.95,1.4) .. (10.93,3.29)   ;
\draw [shift={(394.83,122.5)}, rotate = 359.68] [color={rgb, 255:red, 0; green, 0; blue, 0 }  ][fill={rgb, 255:red, 0; green, 0; blue, 0 }  ][line width=0.75]      (0, 0) circle [x radius= 3.35, y radius= 3.35]   ;
%Straight Lines [id:da054162623086649986] 
\draw    (354.42,223.83) -- (440.33,224) ;
\draw [shift={(442.33,224)}, rotate = 180.11] [color={rgb, 255:red, 0; green, 0; blue, 0 }  ][line width=0.75]    (10.93,-3.29) .. controls (6.95,-1.4) and (3.31,-0.3) .. (0,0) .. controls (3.31,0.3) and (6.95,1.4) .. (10.93,3.29)   ;
\draw [shift={(354.42,223.83)}, rotate = 0.11] [color={rgb, 255:red, 0; green, 0; blue, 0 }  ][fill={rgb, 255:red, 0; green, 0; blue, 0 }  ][line width=0.75]      (0, 0) circle [x radius= 3.35, y radius= 3.35]   ;
%Straight Lines [id:da8565551315819022] 
\draw    (442.33,224) -- (396.54,156.74) ;
\draw [shift={(395.42,155.08)}, rotate = 55.75] [color={rgb, 255:red, 0; green, 0; blue, 0 }  ][line width=0.75]    (10.93,-3.29) .. controls (6.95,-1.4) and (3.31,-0.3) .. (0,0) .. controls (3.31,0.3) and (6.95,1.4) .. (10.93,3.29)   ;
\draw [shift={(442.33,224)}, rotate = 235.75] [color={rgb, 255:red, 0; green, 0; blue, 0 }  ][fill={rgb, 255:red, 0; green, 0; blue, 0 }  ][line width=0.75]      (0, 0) circle [x radius= 3.35, y radius= 3.35]   ;
%Straight Lines [id:da9401688753644867] 
\draw    (395.42,155.08) -- (355.44,222.12) ;
\draw [shift={(354.42,223.83)}, rotate = 300.81] [color={rgb, 255:red, 0; green, 0; blue, 0 }  ][line width=0.75]    (10.93,-3.29) .. controls (6.95,-1.4) and (3.31,-0.3) .. (0,0) .. controls (3.31,0.3) and (6.95,1.4) .. (10.93,3.29)   ;
\draw [shift={(395.42,155.08)}, rotate = 120.81] [color={rgb, 255:red, 0; green, 0; blue, 0 }  ][fill={rgb, 255:red, 0; green, 0; blue, 0 }  ][line width=0.75]      (0, 0) circle [x radius= 3.35, y radius= 3.35]   ;
%Straight Lines [id:da8020512415031483] 
\draw    (66.33,223) -- (154.83,222.5) ;
\draw [shift={(154.83,222.5)}, rotate = 359.68] [color={rgb, 255:red, 0; green, 0; blue, 0 }  ][fill={rgb, 255:red, 0; green, 0; blue, 0 }  ][line width=0.75]      (0, 0) circle [x radius= 3.35, y radius= 3.35]   ;
\draw [shift={(66.33,223)}, rotate = 359.68] [color={rgb, 255:red, 0; green, 0; blue, 0 }  ][fill={rgb, 255:red, 0; green, 0; blue, 0 }  ][line width=0.75]      (0, 0) circle [x radius= 3.35, y radius= 3.35]   ;
%Straight Lines [id:da6624174627662966] 
\draw    (66.33,223) -- (152.83,222.51) ;
\draw [shift={(154.83,222.5)}, rotate = 179.68] [color={rgb, 255:red, 0; green, 0; blue, 0 }  ][line width=0.75]    (10.93,-3.29) .. controls (6.95,-1.4) and (3.31,-0.3) .. (0,0) .. controls (3.31,0.3) and (6.95,1.4) .. (10.93,3.29)   ;
%Straight Lines [id:da37061908568986757] 
\draw    (154.83,222.5) -- (243.33,222) ;
\draw [shift={(243.33,222)}, rotate = 359.68] [color={rgb, 255:red, 0; green, 0; blue, 0 }  ][fill={rgb, 255:red, 0; green, 0; blue, 0 }  ][line width=0.75]      (0, 0) circle [x radius= 3.35, y radius= 3.35]   ;
%Straight Lines [id:da577832864874763] 
\draw    (154.83,222.5) -- (241.33,222.01) ;
\draw [shift={(243.33,222)}, rotate = 179.68] [color={rgb, 255:red, 0; green, 0; blue, 0 }  ][line width=0.75]    (10.93,-3.29) .. controls (6.95,-1.4) and (3.31,-0.3) .. (0,0) .. controls (3.31,0.3) and (6.95,1.4) .. (10.93,3.29)   ;
\draw [shift={(154.83,222.5)}, rotate = 359.68] [color={rgb, 255:red, 0; green, 0; blue, 0 }  ][fill={rgb, 255:red, 0; green, 0; blue, 0 }  ][line width=0.75]      (0, 0) circle [x radius= 3.35, y radius= 3.35]   ;

% Text Node
\draw (28.33,114) node [anchor=north west][inner sep=0.75pt]   [align=left] {(1)};
% Text Node
\draw (271,113.33) node [anchor=north west][inner sep=0.75pt]   [align=left] {(2)};
% Text Node
\draw (60.33,100.67) node [anchor=north west][inner sep=0.75pt]   [align=left] {1};
% Text Node
\draw (148.33,100) node [anchor=north west][inner sep=0.75pt]   [align=left] {2};
% Text Node
\draw (236.33,100) node [anchor=north west][inner sep=0.75pt]   [align=left] {3};
% Text Node
\draw (302.33,100.67) node [anchor=north west][inner sep=0.75pt]   [align=left] {1};
% Text Node
\draw (390.33,100) node [anchor=north west][inner sep=0.75pt]   [align=left] {2};
% Text Node
\draw (478.33,100) node [anchor=north west][inner sep=0.75pt]   [align=left] {3};
% Text Node
\draw (390,134) node [anchor=north west][inner sep=0.75pt]   [align=left] {1};
% Text Node
\draw (334.67,217.33) node [anchor=north west][inner sep=0.75pt]   [align=left] {2};
% Text Node
\draw (448.67,217.33) node [anchor=north west][inner sep=0.75pt]   [align=left] {3};
% Text Node
\draw (62.33,200.67) node [anchor=north west][inner sep=0.75pt]   [align=left] {1};
% Text Node
\draw (150.33,200) node [anchor=north west][inner sep=0.75pt]   [align=left] {2};
% Text Node
\draw (238.33,200) node [anchor=north west][inner sep=0.75pt]   [align=left] {3};
% Text Node
\draw (28.67,211.67) node [anchor=north west][inner sep=0.75pt]   [align=left] {(3)};
% Text Node
\draw (270.67,210.67) node [anchor=north west][inner sep=0.75pt]   [align=left] {(4)};

\end{tikzpicture}
\caption{Non-isomorphic quivers of type $A_3$}
\label{Non-isomorphic quivers of type A3}
\end{figure}
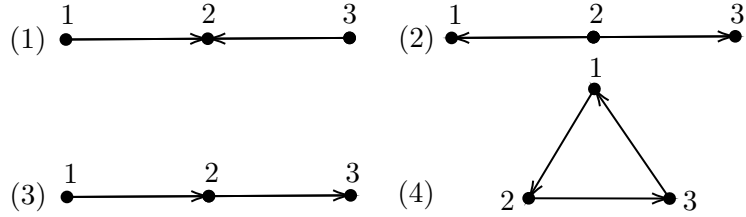

\begin{figure}[htpb]
\tikzset{every picture/.style={line width=0.75pt}} %set default line width to 0.75pt        

\begin{tikzpicture}[x=0.75pt,y=0.75pt,yscale=-1,xscale=1,scale=0.8]
%uncomment if require: \path (0,506); %set diagram left start at 0, and has height of 506

%Straight Lines [id:da8798451956862439] 
\draw    (65.67,124) -- (154.17,123.5) ;
\draw [shift={(154.17,123.5)}, rotate = 359.68] [color={rgb, 255:red, 0; green, 0; blue, 0 }  ][fill={rgb, 255:red, 0; green, 0; blue, 0 }  ][line width=0.75]      (0, 0) circle [x radius= 3.35, y radius= 3.35]   ;
\draw [shift={(65.67,124)}, rotate = 359.68] [color={rgb, 255:red, 0; green, 0; blue, 0 }  ][fill={rgb, 255:red, 0; green, 0; blue, 0 }  ][line width=0.75]      (0, 0) circle [x radius= 3.35, y radius= 3.35]   ;
%Straight Lines [id:da9385912234894725] 
\draw    (65.67,124) -- (152.17,123.51) ;
\draw [shift={(154.17,123.5)}, rotate = 179.68] [color={rgb, 255:red, 0; green, 0; blue, 0 }  ][line width=0.75]    (10.93,-3.29) .. controls (6.95,-1.4) and (3.31,-0.3) .. (0,0) .. controls (3.31,0.3) and (6.95,1.4) .. (10.93,3.29)   ;
%Straight Lines [id:da0659981744072704] 
\draw    (154.17,123.5) -- (242.67,123) ;
\draw [shift={(242.67,123)}, rotate = 359.68] [color={rgb, 255:red, 0; green, 0; blue, 0 }  ][fill={rgb, 255:red, 0; green, 0; blue, 0 }  ][line width=0.75]      (0, 0) circle [x radius= 3.35, y radius= 3.35]   ;
%Straight Lines [id:da9884809577239877] 
\draw    (156.17,123.49) -- (198.01,123.25) -- (242.67,123) ;
\draw [shift={(242.67,123)}, rotate = 359.68] [color={rgb, 255:red, 0; green, 0; blue, 0 }  ][fill={rgb, 255:red, 0; green, 0; blue, 0 }  ][line width=0.75]      (0, 0) circle [x radius= 3.35, y radius= 3.35]   ;
\draw [shift={(154.17,123.5)}, rotate = 359.68] [color={rgb, 255:red, 0; green, 0; blue, 0 }  ][line width=0.75]    (10.93,-3.29) .. controls (6.95,-1.4) and (3.31,-0.3) .. (0,0) .. controls (3.31,0.3) and (6.95,1.4) .. (10.93,3.29)   ;
%Straight Lines [id:da7384457480835956] 
\draw    (306.33,123) -- (394.83,122.5) ;
\draw [shift={(394.83,122.5)}, rotate = 359.68] [color={rgb, 255:red, 0; green, 0; blue, 0 }  ][fill={rgb, 255:red, 0; green, 0; blue, 0 }  ][line width=0.75]      (0, 0) circle [x radius= 3.35, y radius= 3.35]   ;
\draw [shift={(306.33,123)}, rotate = 359.68] [color={rgb, 255:red, 0; green, 0; blue, 0 }  ][fill={rgb, 255:red, 0; green, 0; blue, 0 }  ][line width=0.75]      (0, 0) circle [x radius= 3.35, y radius= 3.35]   ;
%Straight Lines [id:da6562309069221836] 
\draw    (306.33,123) -- (392.83,122.51) ;
\draw [shift={(394.83,122.5)}, rotate = 179.68] [color={rgb, 255:red, 0; green, 0; blue, 0 }  ][line width=0.75]    (10.93,-3.29) .. controls (6.95,-1.4) and (3.31,-0.3) .. (0,0) .. controls (3.31,0.3) and (6.95,1.4) .. (10.93,3.29)   ;
%Straight Lines [id:da6060683139717054] 
\draw    (394.83,122.5) -- (483.33,122) ;
\draw [shift={(483.33,122)}, rotate = 359.68] [color={rgb, 255:red, 0; green, 0; blue, 0 }  ][fill={rgb, 255:red, 0; green, 0; blue, 0 }  ][line width=0.75]      (0, 0) circle [x radius= 3.35, y radius= 3.35]   ;
%Straight Lines [id:da19539941049215048] 
\draw    (394.83,122.5) -- (481.33,122.01) ;
\draw [shift={(483.33,122)}, rotate = 179.68] [color={rgb, 255:red, 0; green, 0; blue, 0 }  ][line width=0.75]    (10.93,-3.29) .. controls (6.95,-1.4) and (3.31,-0.3) .. (0,0) .. controls (3.31,0.3) and (6.95,1.4) .. (10.93,3.29)   ;
\draw [shift={(394.83,122.5)}, rotate = 359.68] [color={rgb, 255:red, 0; green, 0; blue, 0 }  ][fill={rgb, 255:red, 0; green, 0; blue, 0 }  ][line width=0.75]      (0, 0) circle [x radius= 3.35, y radius= 3.35]   ;
%Straight Lines [id:da4312913443037617] 
\draw    (547.42,146.83) -- (633.33,147) ;
\draw [shift={(635.33,147)}, rotate = 180.11] [color={rgb, 255:red, 0; green, 0; blue, 0 }  ][line width=0.75]    (10.93,-3.29) .. controls (6.95,-1.4) and (3.31,-0.3) .. (0,0) .. controls (3.31,0.3) and (6.95,1.4) .. (10.93,3.29)   ;
\draw [shift={(547.42,146.83)}, rotate = 0.11] [color={rgb, 255:red, 0; green, 0; blue, 0 }  ][fill={rgb, 255:red, 0; green, 0; blue, 0 }  ][line width=0.75]      (0, 0) circle [x radius= 3.35, y radius= 3.35]   ;
%Straight Lines [id:da20077516426069542] 
\draw    (635.33,147) -- (589.54,79.74) ;
\draw [shift={(588.42,78.08)}, rotate = 55.75] [color={rgb, 255:red, 0; green, 0; blue, 0 }  ][line width=0.75]    (10.93,-3.29) .. controls (6.95,-1.4) and (3.31,-0.3) .. (0,0) .. controls (3.31,0.3) and (6.95,1.4) .. (10.93,3.29)   ;
\draw [shift={(635.33,147)}, rotate = 235.75] [color={rgb, 255:red, 0; green, 0; blue, 0 }  ][fill={rgb, 255:red, 0; green, 0; blue, 0 }  ][line width=0.75]      (0, 0) circle [x radius= 3.35, y radius= 3.35]   ;
%Straight Lines [id:da08531872917150285] 
\draw    (588.42,78.08) -- (548.44,145.12) ;
\draw [shift={(547.42,146.83)}, rotate = 300.81] [color={rgb, 255:red, 0; green, 0; blue, 0 }  ][line width=0.75]    (10.93,-3.29) .. controls (6.95,-1.4) and (3.31,-0.3) .. (0,0) .. controls (3.31,0.3) and (6.95,1.4) .. (10.93,3.29)   ;
\draw [shift={(588.42,78.08)}, rotate = 120.81] [color={rgb, 255:red, 0; green, 0; blue, 0 }  ][fill={rgb, 255:red, 0; green, 0; blue, 0 }  ][line width=0.75]      (0, 0) circle [x radius= 3.35, y radius= 3.35]   ;

% Text Node
\draw (28.33,114) node [anchor=north west][inner sep=0.75pt]   [align=left] {(1)};
% Text Node
\draw (271,113.33) node [anchor=north west][inner sep=0.75pt]   [align=left] {(2)};
% Text Node
\draw (513.67,112.67) node [anchor=north west][inner sep=0.75pt]   [align=left] {(3)};
% Text Node
\draw (60.33,100.67) node [anchor=north west][inner sep=0.75pt]   [align=left] {1};
% Text Node
\draw (148.33,100) node [anchor=north west][inner sep=0.75pt]   [align=left] {2};
% Text Node
\draw (236.33,100) node [anchor=north west][inner sep=0.75pt]   [align=left] {3};
% Text Node
\draw (302.33,100.67) node [anchor=north west][inner sep=0.75pt]   [align=left] {1};
% Text Node
\draw (390.33,100) node [anchor=north west][inner sep=0.75pt]   [align=left] {2};
% Text Node
\draw (478.33,100) node [anchor=north west][inner sep=0.75pt]   [align=left] {3};
% Text Node
\draw (583,57) node [anchor=north west][inner sep=0.75pt]   [align=left] {1};
% Text Node
\draw (527.67,140.33) node [anchor=north west][inner sep=0.75pt]   [align=left] {2};
% Text Node
\draw (641.67,140.33) node [anchor=north west][inner sep=0.75pt]   [align=left] {3};

\end{tikzpicture}
\caption{Three reduced cases of type $A_3$}
\label{three cases}
\end{figure}
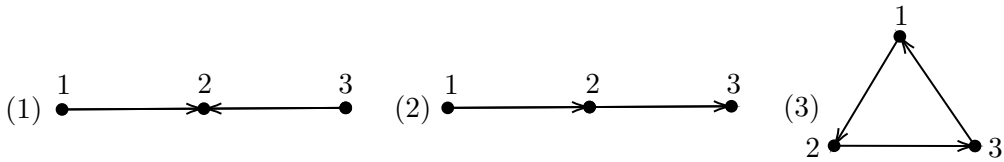

Here, we list all the clusters of these three reduced cases, see \Cref{table inward}, \Cref{table straightfoward} and \Cref{table (3)}. For convenience, we respectively call the case $(1)$, case $(2)$ and case $(3)$ of \emph{Inward type}, \emph{Straightforward type}, and \emph{Cyclic type}.
\begin{remark}
	In \Cref{Generalized associahedra}, it is a generalized associahedra of type 
$A_3$, which is also called a \emph{Stasheff polytope}. For the cluster algebras of type $A_3$, any cluster can be realized as a triangulation of a $6$-gon, see also \cite{FST08, Wil14}. In fact, the generalized associahedra is the dual complex of the cluster complex.
\end{remark}

\begin{table}[h]
\centering
\caption{Clusters of type $A_3$ for the reduced case $(1)$}
\begin{tabularx}{\textwidth}{cX cX}
\toprule
No. & Unlabeled cluster & No. &  Unlabeled cluster \\
\midrule
1 & $(x_{1}, x_{2}, x_{3})$ 
  & 8 & $(\frac{x_1x_3+x_2+1}{x_2x_3},\frac{x_1x_3+(x_2+1)^2}{x_1x_2x_3},\frac{x_2+1}{x_3})$ \\[4pt]

2 & $(\frac{x_1x_3+x_2+1}{x_1x_2},\frac{x_1x_3+1}{x_2},\frac{x_1x_3+x_2+1}{x_2x_3})$ 
  & 9 & $(\frac{x_2+1}{x_1},x_2,x_3)$ \\[4pt]

3 & $(\frac{x_2+1}{x_3},\frac{x_1x_3+(x_2+1)^2}{x_1x_2x_3},\frac{x_2+1}{x_1})$ 
  & 10 & $(x_1,\frac{x_1x_3+1}{x_2},\frac{x_1x_3+x_2+1}{x_2x_3})$ \\[4pt]

4 & $(x_1,\frac{x_1x_3+1}{x_2},x_3)$ 
  & 11 & $(\frac{x_1x_3+x_2+1}{x_1x_2},\frac{x_1x_3+(x_2+1)^2}{x_1x_2x_3},\frac{x_2+1}{x_1})$ \\[4pt]

5 & $(\frac{x_1x_3+x_2+1}{x_1x_2},\frac{x_1x_3+(x_2+1)^2}{x_1x_2x_3},\frac{x_1x_3+x_2+1}{x_2x_3})$ 
  & 12 & $(\frac{x_2+1}{x_3},x_2,x_1)$ \\[4pt]

6 & $(\frac{x_2+1}{x_3},x_2,\frac{x_2+1}{x_1})$ 
  & 13 & $(\frac{x_2+1}{x_1},\frac{x_1x_3+x_2+1}{x_1x_2},x_3)$ \\[4pt]

7 & $(x_3,\frac{x_1x_3+1}{x_2},\frac{x_1x_3+x_2+1}{x_1x_2})$ 
  & 14 & $(\frac{x_2+1}{x_3},\frac{x_1x_3+x_2+1}{x_2x_3},x_1)$ \\[4pt]
\bottomrule
\end{tabularx}
\label{table inward}
\end{table}

\begin{table}[h]
\centering
\caption{Clusters of type $A_3$ for the reduced case $(2)$}
\begin{tabularx}{\textwidth}{cX cX}
\toprule
No. & Unlabeled cluster & No. &  Unlabeled cluster \\
\midrule
1 & $(x_{1}, x_{2}, x_{3})$ 
  & 8 & $(x_1,\frac{x_1+x_3}{x_2},\frac{x_1+x_3+x_1x_2}{x_2x_3})$ \\[4pt]

2 & $(\frac{1+x_2}{x_1},x_2,x_3)$ 
  & 9 & $(x_1,\frac{x_1+x_3+x_1x_2}{x_2x_3},\frac{1+x_2}{x_3})$ \\[4pt]

3 & $(x_1,\frac{x_1+x_3}{x_2},x_3)$ 
  & 10 & $(\frac{1+x_2}{x_1},\frac{x_1+x_3+x_2x_3}{x_1x_2},\frac{(1+x_2)(x_1+x_3)}{x_1x_2x_3})$ \\[4pt]

4 & $(x_1,x_2,\frac{1+x_2}{x_3})$ 
  & 11 & $(\frac{1+x_2}{x_1},\frac{(1+x_2)(x_1+x_3)}{x_1x_2x_3},\frac{1+x_2}{x_3})$ \\[4pt]

5 & $(\frac{1+x_2}{x_1},\frac{x_1+x_3+x_2x_3}{x_1x_2},x_3)$ 
  & 12 & $(\frac{x_1+x_3+x_2x_3}{x_1x_2},\frac{x_1+x_3}{x_2},\frac{(1+x_2)(x_1+x_3)}{x_1x_2x_3})$ \\[4pt]

6 & $(\frac{1+x_2}{x_1},x_2,\frac{1+x_2}{x_3})$ 
  & 13 & $(\frac{(1+x_2)(x_1+x_3)}{x_1x_2x_3},\frac{x_1+x_3}{x_2},\frac{x_1+x_3+x_1x_2}{x_2x_3})$ \\[4pt]

7 & $(\frac{x_1+x_3+x_2x_3}{x_1x_2},\frac{x_1+x_3}{x_2},x_3)$ 
  & 14 & $(\frac{(1+x_2)(x_1+x_3)}{x_1x_2x_3},\frac{x_1+x_3+x_1x_2}{x_2x_3},\frac{1+x_2}{x_3})$ \\[4pt]
\bottomrule
\end{tabularx}
\label{table straightfoward}
\end{table}

\begin{table}[h]
\centering
\caption{Clusters of type $A_3$ for the case $(3)$}
\begin{tabularx}{\textwidth}{cX cX}
\toprule
No. & Unlabeled cluster & No. &  Unlabeled cluster \\
\midrule
1 & $(x_{1}, x_{2}, x_{3})$ 
  & 8 & $(\frac{x_1+x_3}{x_2},x_1,\frac{x_1+x_2+x_3}{x_2x_3})$ \\[4pt]

2 & $(\frac{x_2+x_3}{x_1},x_2,x_3)$ 
  & 9 & $(\frac{x_1+x_2}{x_3},x_1,\frac{x_1+x_2+x_3}{x_2x_3})$ \\[4pt]

3 & $(x_1,\frac{x_1+x_3}{x_2},x_3)$ 
  & 10 & $(\frac{x_1+x_2}{x_3},\frac{x_1+x_2+x_3}{x_1x_3},\frac{x_1+x_2+x_3}{x_2x_3})$ \\[4pt]

4 & $(x_1,x_2,\frac{x_1+x_2}{x_3})$ 
  & 11 & $(\frac{x_1+x_2}{x_3},\frac{x_1+x_2+x_3}{x_1x_3},x_2)$ \\[4pt]

5 & $(\frac{x_2+x_3}{x_1},\frac{x_1+x_2+x_3}{x_1x_2},x_3)$ 
  & 12 & $(\frac{x_2+x_3}{x_1},\frac{x_1+x_2+x_3}{x_1x_3},x_2)$ \\[4pt]

6 & $(\frac{x_1+x_3}{x_2},\frac{x_1+x_2+x_3}{x_1x_2},x_3)$ 
  & 13 & $(\frac{x_2+x_3}{x_1},\frac{x_1+x_2+x_3}{x_1x_3},\frac{x_1+x_2+x_3}{x_1x_2})$ \\[4pt]

7 & $(\frac{x_1+x_3}{x_2},\frac{x_1+x_2+x_3}{x_1x_2},\frac{x_1+x_2+x_3}{x_2x_3})$ 
  & 14 & $(\frac{x_1+x_2+x_3}{x_1x_2},\frac{x_1+x_2+x_3}{x_1x_3},\frac{x_1+x_2+x_3}{x_2x_3})$ \\[4pt]
\bottomrule
\end{tabularx}
\label{table (3)}
\end{table}

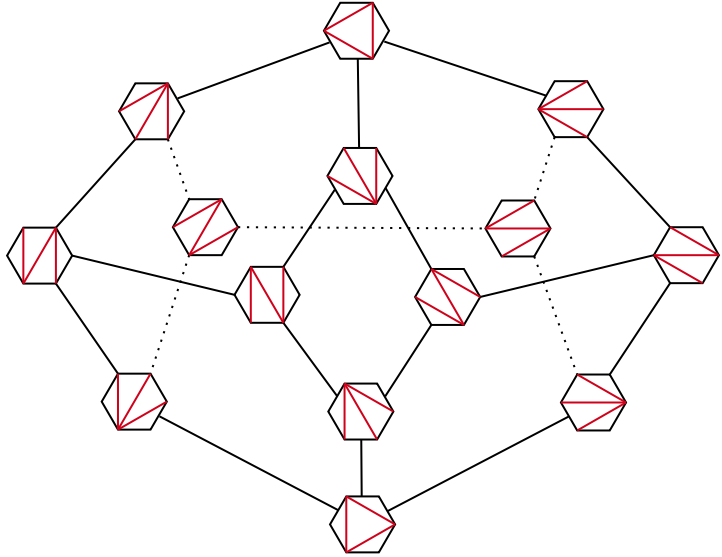
\begin{figure}[htpb]
\tikzset{every picture/.style={line width=0.75pt}} %set default line width to 0.75pt        
\begin{tikzpicture}[x=0.75pt,y=0.75pt,yscale=-1,xscale=1,scale=0.65]
%uncomment if require: \path (0,614); %set diagram left start at 0, and has height of 614

%Shape: Regular Polygon [id:dp28793542400763805] 
\draw   (361.33,49.33) -- (348.83,70.98) -- (323.83,70.98) -- (311.33,49.33) -- (323.83,27.68) -- (348.83,27.68) -- cycle ;
%Shape: Regular Polygon [id:dp9859461570579019] 
\draw   (204,111.67) -- (191.5,133.32) -- (166.5,133.32) -- (154,111.67) -- (166.5,90.02) -- (191.5,90.02) -- cycle ;
%Shape: Regular Polygon [id:dp6837472735106613] 
\draw   (526,110) -- (513.5,131.65) -- (488.5,131.65) -- (476,110) -- (488.5,88.35) -- (513.5,88.35) -- cycle ;
%Shape: Regular Polygon [id:dp19342333480092821] 
\draw   (364,161.67) -- (351.5,183.32) -- (326.5,183.32) -- (314,161.67) -- (326.5,140.02) -- (351.5,140.02) -- cycle ;
%Shape: Regular Polygon [id:dp0850686203205736] 
\draw   (292.67,253.67) -- (280.17,275.32) -- (255.17,275.32) -- (242.67,253.67) -- (255.17,232.02) -- (280.17,232.02) -- cycle ;
%Shape: Regular Polygon [id:dp911044467565989] 
\draw   (431.33,255.33) -- (418.83,276.98) -- (393.83,276.98) -- (381.33,255.33) -- (393.83,233.68) -- (418.83,233.68) -- cycle ;
%Shape: Regular Polygon [id:dp39589995644390663] 
\draw   (364.67,344) -- (352.17,365.65) -- (327.17,365.65) -- (314.67,344) -- (327.17,322.35) -- (352.17,322.35) -- cycle ;
%Shape: Regular Polygon [id:dp5348317286336324] 
\draw   (245.33,201.33) -- (232.83,222.98) -- (207.83,222.98) -- (195.33,201.33) -- (207.83,179.68) -- (232.83,179.68) -- cycle ;
%Shape: Regular Polygon [id:dp49667190549962403] 
\draw   (485.33,202.33) -- (472.83,223.98) -- (447.83,223.98) -- (435.33,202.33) -- (447.83,180.68) -- (472.83,180.68) -- cycle ;
%Shape: Regular Polygon [id:dp1570506866669461] 
\draw   (366,431.33) -- (353.5,452.98) -- (328.5,452.98) -- (316,431.33) -- (328.5,409.68) -- (353.5,409.68) -- cycle ;
%Shape: Regular Polygon [id:dp5218694074422687] 
\draw   (118,223.33) -- (105.5,244.98) -- (80.5,244.98) -- (68,223.33) -- (80.5,201.68) -- (105.5,201.68) -- cycle ;
%Shape: Regular Polygon [id:dp4923945772392223] 
\draw   (614.67,223) -- (602.17,244.65) -- (577.17,244.65) -- (564.67,223) -- (577.17,201.35) -- (602.17,201.35) -- cycle ;
%Shape: Regular Polygon [id:dp458669379316049] 
\draw   (190.67,336.33) -- (178.17,357.98) -- (153.17,357.98) -- (140.67,336.33) -- (153.17,314.68) -- (178.17,314.68) -- cycle ;
%Shape: Regular Polygon [id:dp06485246668642441] 
\draw   (543.33,337) -- (530.83,358.65) -- (505.83,358.65) -- (493.33,337) -- (505.83,315.35) -- (530.83,315.35) -- cycle ;
%Straight Lines [id:da9961846115123152] 
\draw    (199,101.67) -- (315.67,58.33) ;
%Straight Lines [id:da9271593807348066] 
\draw    (357.67,57.67) -- (482.33,99.67) ;
%Straight Lines [id:da24077734811130624] 
\draw  [dash pattern={on 0.84pt off 2.51pt}]  (191.5,133.32) -- (207.83,179.68) ;
%Straight Lines [id:da3789328745676128] 
\draw  [dash pattern={on 0.84pt off 2.51pt}]  (488.5,131.65) -- (472.83,180.68) ;
%Straight Lines [id:da5949722051709764] 
\draw  [dash pattern={on 0.84pt off 2.51pt}]  (245.33,201.33) -- (435.33,202.33) ;
%Straight Lines [id:da6469848400776123] 
\draw    (319.67,172.33) -- (280.17,232.02) ;
%Straight Lines [id:da02090490298318959] 
\draw    (358.67,171) -- (393.83,233.68) ;
%Straight Lines [id:da04591576891493332] 
\draw    (280.17,275.32) -- (321.67,332.33) ;
%Straight Lines [id:da8130748964761626] 
\draw    (393.83,276.98) -- (358.33,332.33) ;
%Straight Lines [id:da3655289377660541] 
\draw  [dash pattern={on 0.84pt off 2.51pt}]  (207.83,222.98) -- (178.17,314.68) ;
%Straight Lines [id:da830678393145896] 
\draw  [dash pattern={on 0.84pt off 2.51pt}]  (472.83,223.98) -- (505.83,315.35) ;
%Straight Lines [id:da3843225848977698] 
\draw    (166.5,133.32) -- (105.5,201.68) ;
%Straight Lines [id:da9140685264791069] 
\draw    (105.5,244.98) -- (153.17,314.68) ;
%Straight Lines [id:da07054210276549555] 
\draw    (513.5,131.65) -- (577.17,201.35) ;
%Straight Lines [id:da34234837709772437] 
\draw    (577.17,244.65) -- (530.83,315.35) ;
%Straight Lines [id:da8259705651206003] 
\draw    (340,365.33) -- (340.33,409.67) ;
%Straight Lines [id:da8618813136261955] 
\draw    (118,223.33) -- (242.67,253.67) ;
%Straight Lines [id:da029381760985648264] 
\draw    (431.33,255.33) -- (564.67,223) ;
%Straight Lines [id:da6018689636042475] 
\draw    (185,347.67) -- (322.33,420.33) ;
%Straight Lines [id:da5435693001590655] 
\draw    (360.33,421) -- (499.67,347.67) ;
%Straight Lines [id:da07387044900735051] 
\draw    (337.33,70.67) -- (338.33,139.67) ;
%Straight Lines [id:da43703109007575824] 
\draw [color={rgb, 255:red, 208; green, 2; blue, 27 }  ,draw opacity=1 ]   (351.5,140.02) -- (351.5,183.32) ;
%Straight Lines [id:da6509948116784292] 
\draw [color={rgb, 255:red, 208; green, 2; blue, 27 }  ,draw opacity=1 ]   (326.5,140.02) -- (351.5,183.32) ;
%Straight Lines [id:da7625319032205427] 
\draw [color={rgb, 255:red, 208; green, 2; blue, 27 }  ,draw opacity=1 ]   (314,161.67) -- (351.5,183.32) ;
%Straight Lines [id:da0641038161261609] 
\draw [color={rgb, 255:red, 208; green, 2; blue, 27 }  ,draw opacity=1 ]   (255.17,232.02) -- (255.17,275.32) ;
%Straight Lines [id:da2089271539498927] 
\draw [color={rgb, 255:red, 208; green, 2; blue, 27 }  ,draw opacity=1 ]   (255.17,232.02) -- (280.17,275.32) ;
%Straight Lines [id:da417450900911462] 
\draw [color={rgb, 255:red, 208; green, 2; blue, 27 }  ,draw opacity=1 ]   (280.17,232.02) -- (280.17,275.32) ;
%Straight Lines [id:da19199636074149462] 
\draw [color={rgb, 255:red, 208; green, 2; blue, 27 }  ,draw opacity=1 ]   (381.33,255.33) -- (418.83,276.98) ;
%Straight Lines [id:da8149190220942439] 
\draw [color={rgb, 255:red, 208; green, 2; blue, 27 }  ,draw opacity=1 ]   (393.83,233.68) -- (418.83,276.98) ;
%Straight Lines [id:da4378352855425698] 
\draw [color={rgb, 255:red, 208; green, 2; blue, 27 }  ,draw opacity=1 ]   (393.83,233.68) -- (431.33,255.33) ;
%Straight Lines [id:da6802707799379937] 
\draw [color={rgb, 255:red, 208; green, 2; blue, 27 }  ,draw opacity=1 ]   (327.17,322.35) -- (364.67,344) ;
%Straight Lines [id:da2063728028830486] 
\draw [color={rgb, 255:red, 208; green, 2; blue, 27 }  ,draw opacity=1 ]   (327.17,322.35) -- (352.17,365.65) ;
%Straight Lines [id:da8398844777534319] 
\draw [color={rgb, 255:red, 208; green, 2; blue, 27 }  ,draw opacity=1 ]   (327.17,322.35) -- (327.17,365.65) ;
%Straight Lines [id:da34066429610119064] 
\draw [color={rgb, 255:red, 208; green, 2; blue, 27 }  ,draw opacity=1 ]   (311.33,49.33) -- (348.83,70.98) ;
%Straight Lines [id:da8317739473913717] 
\draw [color={rgb, 255:red, 208; green, 2; blue, 27 }  ,draw opacity=1 ]   (348.83,27.68) -- (348.83,70.98) ;
%Straight Lines [id:da06680075364966598] 
\draw [color={rgb, 255:red, 208; green, 2; blue, 27 }  ,draw opacity=1 ]   (348.83,27.68) -- (311.33,49.33) ;
%Straight Lines [id:da2854105297390105] 
\draw [color={rgb, 255:red, 208; green, 2; blue, 27 }  ,draw opacity=1 ]   (191.5,90.02) -- (154,111.67) ;
%Straight Lines [id:da1878335605040493] 
\draw [color={rgb, 255:red, 208; green, 2; blue, 27 }  ,draw opacity=1 ]   (191.5,90.02) -- (166.5,133.32) ;
%Straight Lines [id:da1058752999415642] 
\draw [color={rgb, 255:red, 208; green, 2; blue, 27 }  ,draw opacity=1 ]   (191.5,90.02) -- (191.5,133.32) ;
%Straight Lines [id:da46794441835433587] 
\draw [color={rgb, 255:red, 208; green, 2; blue, 27 }  ,draw opacity=1 ]   (476,110) -- (513.5,131.65) ;
%Straight Lines [id:da1469562419920658] 
\draw [color={rgb, 255:red, 208; green, 2; blue, 27 }  ,draw opacity=1 ]   (476,110) -- (526,110) ;
%Straight Lines [id:da12924404070316498] 
\draw [color={rgb, 255:red, 208; green, 2; blue, 27 }  ,draw opacity=1 ]   (476,110) -- (513.5,88.35) ;
%Straight Lines [id:da13246888624753517] 
\draw [color={rgb, 255:red, 208; green, 2; blue, 27 }  ,draw opacity=1 ]   (80.5,201.68) -- (80.5,244.98) ;
%Straight Lines [id:da05483601722205145] 
\draw [color={rgb, 255:red, 208; green, 2; blue, 27 }  ,draw opacity=1 ]   (105.5,201.68) -- (105.5,244.98) ;
%Straight Lines [id:da07379468365789865] 
\draw [color={rgb, 255:red, 208; green, 2; blue, 27 }  ,draw opacity=1 ]   (80.5,244.98) -- (105.5,201.68) ;
%Straight Lines [id:da5021614308232994] 
\draw [color={rgb, 255:red, 208; green, 2; blue, 27 }  ,draw opacity=1 ]   (153.17,357.98) -- (190.67,336.33) ;
%Straight Lines [id:da29868064650944015] 
\draw [color={rgb, 255:red, 208; green, 2; blue, 27 }  ,draw opacity=1 ]   (153.17,357.98) -- (178.17,314.68) ;
%Straight Lines [id:da7812959416713786] 
\draw [color={rgb, 255:red, 208; green, 2; blue, 27 }  ,draw opacity=1 ]   (153.17,357.98) -- (153.17,314.68) ;
%Straight Lines [id:da22115812496851217] 
\draw [color={rgb, 255:red, 208; green, 2; blue, 27 }  ,draw opacity=1 ]   (195.33,201.33) -- (232.83,179.68) ;
%Straight Lines [id:da2014277114882237] 
\draw [color={rgb, 255:red, 208; green, 2; blue, 27 }  ,draw opacity=1 ]   (232.83,179.68) -- (207.83,222.98) ;
%Straight Lines [id:da9924436202450782] 
\draw [color={rgb, 255:red, 208; green, 2; blue, 27 }  ,draw opacity=1 ]   (207.83,222.98) -- (245.33,201.33) ;
%Straight Lines [id:da07136136429823225] 
\draw [color={rgb, 255:red, 208; green, 2; blue, 27 }  ,draw opacity=1 ]   (472.83,180.68) -- (435.33,202.33) ;
%Straight Lines [id:da403004228240853] 
\draw [color={rgb, 255:red, 208; green, 2; blue, 27 }  ,draw opacity=1 ]   (435.33,202.33) -- (485.33,202.33) ;
%Straight Lines [id:da2693900339479214] 
\draw [color={rgb, 255:red, 208; green, 2; blue, 27 }  ,draw opacity=1 ]   (485.33,202.33) -- (447.83,223.98) ;
%Straight Lines [id:da7805824960979928] 
\draw [color={rgb, 255:red, 208; green, 2; blue, 27 }  ,draw opacity=1 ]   (577.17,201.35) -- (614.67,223) ;
%Straight Lines [id:da5193966545497004] 
\draw [color={rgb, 255:red, 208; green, 2; blue, 27 }  ,draw opacity=1 ]   (564.67,223) -- (602.17,244.65) ;
%Straight Lines [id:da32561238556000016] 
\draw [color={rgb, 255:red, 208; green, 2; blue, 27 }  ,draw opacity=1 ]   (614.67,223) -- (564.67,223) ;
%Straight Lines [id:da5498173675436541] 
\draw [color={rgb, 255:red, 208; green, 2; blue, 27 }  ,draw opacity=1 ]   (366,431.33) -- (328.5,452.98) ;
%Straight Lines [id:da7503715858614713] 
\draw [color={rgb, 255:red, 208; green, 2; blue, 27 }  ,draw opacity=1 ]   (328.5,409.68) -- (366,431.33) ;
%Straight Lines [id:da04962386218265069] 
\draw [color={rgb, 255:red, 208; green, 2; blue, 27 }  ,draw opacity=1 ]   (328.5,452.98) -- (328.5,409.68) ;
%Straight Lines [id:da2985773688807074] 
\draw [color={rgb, 255:red, 208; green, 2; blue, 27 }  ,draw opacity=1 ]   (543.33,337) -- (493.33,337) ;
%Straight Lines [id:da4863966098736321] 
\draw [color={rgb, 255:red, 208; green, 2; blue, 27 }  ,draw opacity=1 ]   (543.33,337) -- (505.83,315.35) ;
%Straight Lines [id:da8474256519907136] 
\draw [color={rgb, 255:red, 208; green, 2; blue, 27 }  ,draw opacity=1 ]   (543.33,337) -- (505.83,358.65) ;

\end{tikzpicture}
\caption{The generalized associahedra of type $A_3$}
\label{Generalized associahedra}
\end{figure}

\section{Log-concavity and unimodality of cluster monomials of type $A_3$}\label{sec: main}
In this section, we focus on proving \Cref{An conj} and the unimodality for all the cluster algebras of type $A_3$. Then, we give a refined conjecture of \Cref{An conj}.
\subsection{Main theorem} At this point, we are ready to exhibit our main result.
\begin{theorem}\label{thm: main result}
	The cluster monomials of type $A_3$ are log-concave and unimodal.
\end{theorem} The main proof of \Cref{thm: main result} will be presented in the next subsection. Firstly, we turn to an illustrative example as follows. 
\begin{example}\label{ex: log-concavity of A3}
	Take a cluster monomial with respect to No. $13$ cluster in \Cref{table straightfoward}, and with the exponent vector $(1,1,1)$. Then, this cluster monomial can be written as  
	\begin{align}
		& \frac{(1+x_2)(x_1+x_3)}{x_1x_2x_3}\cdot \frac{x_1+x_3}{x_2}\cdot\frac{x_1+x_3+x_1x_2}{x_2x_3}\\ =& \frac{x_1^3(1+2x_2+x_2^2)+x_1^2(3+5x_2+2x_2^2)x_3+x_1(3+4x_2+x_2^2)x_3^2+(1+x_2)x_3^3}{x_1x_2^3x_3^2}.
	\end{align} By a direct calculation, we claim that this Laurent polynomial is log-concave and unimodal. 
\end{example}
\subsection{Proof of the main theorem}
In this subsection, we will prove \Cref{thm: main result} by respectively considering three reduced types. According to the Laurent phenomenon and \Cref{same up}, it is sufficient to consider only the numerator of each cluster monomial. 

\subsubsection{Inward type} Note that this type corresponds to \Cref{table inward}. Hence, by direct observation, it is sufficient to prove that $(x_1x_3+1)^a(x_1x_3+x_2+1)^b$ and $(x_1x_3+x_2+1)^a[x_1x_3+(x_2+1)^2]^b(x_2+1)^c$ are log-concave and unimodal for any $a,b,c \in \mbN$.

Firstly, let us consider $(x_1x_3+1)^a(x_1x_3+x_2+1)^b$. Its expansion formula is given by
\begin{align}
		&(x_1x_3+1)^a(x_1x_3+x_2+1)^b\\ =& (x_1x_3+1)^a[x_2+(x_1x_3+1)]^b \\ =&\sum\limits_{i=0}^b\binom{b}{i}x_2^{i}(x_1x_3+1)^{a+b-i} \\ =& 
\sum\limits_{i=0}^b\sum\limits_{j=0}^{a+b-i}\binom{b}{i}\binom{a+b-i}{j}x_1^{j}x_2^{i}x_3^{j}.
\end{align}
If we fix the exponents of $x_2$ and $x_3$ (or $x_1$ and $x_2$), the exponent of $x_1$ (or $x_3$) is fixed automatically. Hence, there are no three consecutive terms associated with the variable $x_1$ (or $x_3$), which satisfies the log-concavity and unimodality in this direction. If we fix the exponents $j$ of $x_1$ and $x_3$, the coefficients of $x_1^{j}x_2^{i-1}x_3^j$, $x_1^{j}x_2^{i}x_3^j$ and $x_1^{j}x_2^{i+1}x_3^j$ are respectively $\binom{b}{i}\binom{a+b-i}{j-1}$, $\binom{b}{i}\binom{a+b-i}{j}$ and $\binom{b}{i}\binom{a+b-i}{j+1}$. By \Cref{lem: binomial1}, we have 
\begin{align}
	\left[\binom{b}{i}\binom{a+b-i}{j}\right]^2\geq \binom{b}{i}\binom{a+b-i}{j-1}\cdot\binom{b}{i}\binom{a+b-i}{j+1}.
\end{align} Hence, the log-concavity and unimodality of $(x_1x_3+1)^a(x_1x_2+x_2+1)^b$ is obtained.

In the following, we concentrate on the much more complicated monomial $(x_1x_3+x_2+1)^a[x_1x_3+(x_2+1)^2]^b(x_2+1)^c$ with $a,b,c\in \mbN$. Beforehand, we need to introduce some preliminary definitions and lemmas. 
\begin{definition}[\emph{Gauss hypergeometric function}]\label{def: hyper}
	The Gauss hypergeometric function ${}_2F_1$ is defined by the power series
\begin{align}
{}_2F_1(a,b;c;z)
=
\sum_{n=0}^{\infty}
\frac{(a)_n (b)_n}{(c)_n}
\frac{z^n}{n!},
\end{align}
where $(q)_n = q(q+1)\cdots(q+n-1)$ denotes the rising Pochhammer symbol, with
$(q)_0 = 1$.
This series converges absolutely for $|z|<1$.
\end{definition}
\begin{lemma}[cf. \cite{OLBC10}]\label{lem: Pfaff}
	The Pfaff transformations hold that 
		\begin{enumerate}
			\item ${}_2F_1(a,b;c;z)=(1-z)^{-a}{}_2F_1(a,c-b;c;\frac{z}{z-1})$, $|\arg(1-z)|<\pi$.
			\item ${}_2F_1(a,b;c;z)=(1-z)^{-b}{}_2F_1(c-a,b;c;\frac{z}{z-1})$, $|\arg(1-z)|<\pi$.
		\end{enumerate}
\end{lemma}
\begin{definition}[\emph{Jacobi polynomial}]\label{lem: Jacobi}
For parameters $\alpha,\beta>-1$, the Jacobi polynomial
$P_n^{(\alpha,\beta)}(x)$ of degree $n$ can be expressed in terms of the Gauss
hypergeometric function as
\begin{align}
P_n^{(\alpha,\beta)}(x)
=
\frac{(\alpha+1)_n}{n!}\,
{}_2F_1(-n, 1+\alpha+\beta+n; \alpha+1; \frac{1-x}{2}).
\end{align}
In particular, since one of the upper parameters equals to $-n$, the hypergeometric
series terminates and reduces to a polynomial of degree $n$.
\end{definition}
\begin{lemma}[cf. \cite{Sze75}]\label{lem: real roots}
	All the zeros of the Jacobi polynomial $P_n^{(\alpha,\beta)}(x)$ are real, simple and lie in the open interval $(-1,1)$.
\end{lemma}
%There is a fact that all the zeros of the Jacobi polynomial $P_n^{(\alpha,\beta)}(x)$ are real, simple and lie in the open interval $(-1,1)$, see \cite{Sze75}.

Now, let us recall the well-known \emph{Newton inequalities} as follows.
\begin{lemma}[cf. \cite{Bre89}, \cite{HLP52}]\label{thm: newton}
	Let $f(x)=\sum_{k=0}^na_kx^k$ be a polynomial with nonnegative coefficients. If all the zeros of $f(x)$ are real, then the coefficient sequence $\{a_0,\dots,a_n\}$ is log-concave.
\end{lemma}
\begin{remark}
	From the perspective of Galois theory, polynomial equations of degrees at most four admit general solutions in radicals, whereas no such formula exists for general equations of degree five or higher. Hence, in \Cref{thm: newton}, we may directly determine the real-rootedness for the polynomial equations of degrees at most four. However, for the higher degree, we need to use more general methods, such as Jacobi polynomials.
\end{remark}
Now, let us consider the expansion formula of $(x_1x_3+x_2+1)^a[x_1x_3+(x_2+1)^2]^b(x_2+1)^c$.
\begin{align}
		&(x_1x_3+x_2+1)^a[x_1x_3+(x_2+1)^2]^b(x_2+1)^c\\ =& \left[\sum\limits_{i=0}^a\binom{a}{i}x_1^{a-i}x_3^{a-i}(1+x_2)^{c+i}\right]\left[\sum\limits_{j=0}^b\binom{b}{j}x_1^{b-j}x_3^{b-j}(1+x_2)^{2j}\right]\\ =& \sum\limits_{i=0}^a\sum\limits_{j=0}^b\binom{a}{i}\binom{b}{j}x_1^{a+b-i-j}x_3^{a+b-i-j}(1+x_2)^{i+2j+c} \\ =& \sum\limits_{i=0}^a\sum\limits_{j=0}^b\sum\limits_{k=0}^{i+2j+c}\binom{a}{i}\binom{b}{j}\binom{i+2j+c}{k}x_1^{a+b-i-j}x_2^{k}x_3^{a+b-i-j}. 
\end{align}
Note that the exponents of $x_1$ and $x_3$ in each single term are the same. Hence, it directly implies the log-concavity and unimodality of this monomial in the direction $x_1$ and $x_3$. Then, we only need to consider the case for $x_2$. If we fix the exponents of $x_1$ and $x_3$, that is $i+j=N$ is fixed, the coefficient of the term $x_1^{a+b-N}x_2^{k}x_3^{a+b-N}$ is given by 
\begin{align}
	S_k\coloneqq \sum_{\substack{0 \le i \le a,\, 0 \le j \le b, \\ i+j=N} }\binom{a}{i}\binom{b}{j}\binom{N+j+c}{k}.
\end{align} It is direct that there is no internal zero of the sequence $\{S_k\}$. In the following, we aim to prove that $S_{k}^2\geq S_{k-1}S_{k+1}$. Note that $S_k$ is the coefficient of $x^k$ in the expansion of the polynomial 
\begin{align}
	\sum_{\substack{0 \le i \le a,\, 0 \le j \le b, \\ i+j=N} }\binom{a}{i}\binom{b}{j}(1+x)^{N+j+c}.  
\end{align} Then, we consider its generating function 
\begin{align}
	F(x)= \sum\limits_{k\geq 0}S_kx^k=(1+x)^{N+c}\sum\limits_{j=0}^N\binom{a}{N-j}\binom{b}{j}(1+x)^j. \label{eq: generating function}
\end{align} Without loss of generality, we might assume that $a\geq N$ and $b\geq N$. Otherwise, the sum in \eqref{eq: generating function} will degenerate to a simpler one since some coefficients become zero. Let $t=1+x$ and 
\begin{align}
	Q(t)=\sum\limits_{j=0}^N\binom{a}{N-j}\binom{b}{j}t^j.\label{eq: cf}
\end{align} Then, we have $F(x)=t^{N+c}Q(t)$. 

We now claim that all the roots of $Q(t)$ are real. Hence, all the roots of $F(x)$ are also real. In fact, we have $\binom{b}{j}=\frac{(-1)^j(-b)_{j}}{j!}$ and 
\begin{align}
	\frac{\binom{a}{N-j}}{\binom{a}{N}}=\frac{N!(a-N)!}{(N-j)!(a-N+j)!}=\frac{(-1)^j(-N)_j}{(a-N+1)_j},
\end{align} which implies that $\binom{a}{N-j}=\binom{a}{N} \frac{(-1)^j(-N)_j}{(a-N+1)_j}.$ With the help of Gauss hypergeometric function and the fact that $(-N)_{j}=0$ for any $j\geq N+1$, we obtain that 
\begin{align}
	Q(t)=\binom{a}{N}\sum\limits_{j=0}^N\frac{(-N)_j (-b)_j}{(a-N+1)_j}
\frac{t^j}{j!}=\binom{a}{N}{}_2F_1(-N,-b;a-N+1;t).
\end{align} By Pfaff transformation $(1)$ of \Cref{lem: Pfaff}, it holds that 
\begin{align}
	Q(t)=\binom{a}{N}(1-t)^N{}_2F_1(-N,a+b-N+1;a-N+1;\frac{t}{t-1}).
\end{align} Take $\alpha=a-N$ and $\beta=b-N$ in the explicit formula of Jacobi polynomial. Then, $\alpha+1=a-N+1$ and $1+\alpha+\beta+N=a+b-N+1$. Let $\frac{t}{t-1} =\frac{1-x}{2}$, so that $x=\frac{1+t}{1-t}$. Thus, we have 
\begin{align}
	Q(t)=\binom{a}{N}(1-t)^N\frac{N!}{(a-N+1)_N}P_N^{(a-N,b-N)}\left(\frac{1+t}{1-t}\right).
\end{align} By \Cref{lem: real roots}, all the roots of $Q(t)$ and $F(x)$ must be real roots. It is worth mentioning that $\frac{1+t}{1-t}\in \mbR$ can imply that $t\in \mbR$. Moreover, since all the coefficients of $F(x)$ are nonnegative, by \Cref{thm: newton}, we conclude that $S_{k}^2\geq S_{k-1}S_{k+1}$.

In conclusion, all the cluster monomials of the inward type cluster algebras of type $A_3$ are log-concave and unimodal.
\subsubsection{Straightforward type} Note that this type corresponds to \Cref{table straightfoward}. Hence, by direct observation, it is sufficient to prove that $(1+x_2)^a(x_1+x_3+x_2x_3)^b[(1+x_2)(x_1+x_3)]^c$ and $(x_1+x_3+x_2x_3)^a(x_1+x_3)^b[(1+x_2)(x_1+x_3)]^c$ are log-concave and unimodal for any $a,b,c \in \mbN$. 

Firstly, let us consider the expansion formula of $(1+x_2)^a(x_1+x_3+x_2x_3)^b[(1+x_2)(x_1+x_3)]^c$, which is given by 
\begin{align}
	& (1+x_2)^a(x_1+x_3+x_2x_3)^b[(1+x_2)(x_1+x_3)]^c \\  =& (1+x_2)^{a+c}\left[\sum_{i=0}^b\binom{b}{i}x_1^i(1+x_2)^{b-i}x_3^{b-i}\right]\left[\sum_{j=0}^{c}\binom{c}{j}x_1^{j}x_3^{c-j}\right] \\ =& \left[\sum_{i=0}^b\binom{b}{i}x_1^i(1+x_2)^{a+b+c-i}x_3^{b-i}\right]\left[\sum_{j=0}^{c}\binom{c}{j}x_1^{j}x_3^{c-j}\right] \\ =& \sum\limits_{i=0}^b\sum\limits_{k=0}^{a+b+c-i}\sum\limits_{j=0}^{c}\binom{b}{i}\binom{c}{j}\binom{a+b+c-i}{k}x_1^{i+j}x_2^{k}x_3^{b+c-i-j}.
\end{align} Note that the sum of the exponents of $x_1$ and $x_3$ is $b+c$. Hence, it directly implies the log-concavity and unimodality of this monomial in the direction $x_1$ and $x_3$. Then, we only need to consider the case for $x_2$. If we fix $i+j=N$, the coefficient of the term $x_1^{N}x_2^{k}x_3^{b+c-N}$ is given by 
\begin{align}
	T_k\coloneqq \sum_{\substack{0 \le i \le b,\, 0 \le j \le c, \\ i+j=N} }\binom{b}{i}\binom{c}{j}\binom{a+b+c-i}{k}.
\end{align} It is direct that there is no internal zero of the sequence $\{T_k\}$. In what follows, we aim to prove that $T_k^2\geq T_{k-1}T_{k+1}$. Note that $T_k$ is the coefficient of $x^k$ in the expansion of the polynomial 
\begin{align}
	\sum_{\substack{0 \le i \le b,\, 0 \le j \le c, \\ i+j=N} }\binom{b}{i}\binom{c}{j}(1+x)^{a+b+c-i}.  
\end{align} Hence, we can consider its generating function
\begin{align}
	F(x)= \sum\limits_{k\geq 0}T_kx^k=(1+x)^{a+b+c}\sum\limits_{i=0}^N\binom{b}{i}\binom{c}{N-i}(1+x)^{-i}. 
\end{align} Moreover, we have $F(x)=(1+x)^{a+b+c}Q(\frac{1}{1+x})$, where 
\begin{align}
	Q(t)=\sum\limits_{i=0}^N\binom{b}{i}\binom{c}{N-i}t^{i}. \end{align}
	By arguments analogous to those for the inward type (cf. \eqref{eq: cf}), we  conclude that all the roots of $F(x)$ are real, with the help of Jacobi polynomials. Hence, by \Cref{thm: newton}, we have $T_k^2\geq T_{k-1}T_{k+1}$, which implies the log-concavity and unimodality of the first case. 
	
Now, we turn to another monomial $(x_1+x_3+x_2x_3)^a(x_1+x_3)^b[(1+x_2)(x_1+x_3)]^c$. Its expansion formula is given by 
\begin{align}
	&(x_1+x_3+x_2x_3)^a(x_1+x_3)^b[(1+x_2)(x_1+x_3)]^c \\ =& \left[\sum_{i=0}^a\binom{a}{i}(x_1+x_3)^{b+c+i}x_2^{a-i}x_3^{a-i}\right]\left[\sum_{j=0}^{c}\binom{c}{j}x_2^{j}\right] \\ =& \sum\limits_{i=0}^a\sum\limits_{k=0}^{b+c+i}\sum\limits_{j=0}^{c}\binom{a}{i}\binom{b+c+i}{k}\binom{c}{j}x_1^{k}x_2^{a-i+j}x_3^{a+b+c-k}.
\end{align} Note that the sum of the exponents of $x_1$ and $x_3$ is $a+b+c$. Hence, it directly implies the log-concavity and unimodality of this monomial in the direction $x_1$ and $x_3$. Then, we only need to consider the case for $x_2$. Fix the exponent $k$ of $x_1$ and then the exponent of $x_3$ is also determined. There are two possible cases to be discussed as follows. 
\begin{enumerate}[leftmargin=2em]
	\item 
If $0\leq k\leq b+c$, the associated terms are 
\begin{align}
	\sum\limits_{i=0}^a\sum\limits_{j=0}^{c}\binom{a}{i}\binom{b+c+i}{k}\binom{c}{j}x_1^{k}x_2^{a-i+j}x_3^{a+b+c-k}.
\end{align} We consider the coefficient corresponding to each exponent of $x_2$. Without loss of generality, we might assume that $a\geq c$. Let $h$ denote the exponent of $x_2$. When $0\leq h\leq c$, there are $h+1$ sum terms in the coefficient of $x_1^{k}x_2^{h}x_3^{a+b+c-k}$, which is given by 
\begin{align}
	\alpha_h &=\sum\limits_{l=0}^h\binom{a}{a-h+l}\binom{a+b+c-h+l}{k}\binom{c}{l}\\ &=\sum\limits_{l=0}^h\binom{a}{h-l}\binom{a+b+c-h+l}{k}\binom{c}{l}.
\end{align} When $c\leq h\leq a$, there are always $c+1$ terms in the coefficient of $x_1^{k}x_2^{h}x_3^{a+b+c-k}$, which is given by 
\begin{align}
	\beta_h & =\sum\limits_{l=0}^c\binom{a}{a-h+l}\binom{a+b+c-h+l}{k}\binom{c}{l}\\ &=\sum\limits_{l=0}^c\binom{a}{h-l}\binom{a+b+c-h+l}{k}\binom{c}{l}.
\end{align} 
When $a\leq h\leq a+c$, there are $a+c-h+1$ sum terms in the coefficient of $x_1^{k}x_2^{h}x_3^{a+b+c-k}$, which is given by 
\begin{align}
	\gamma_h=\sum\limits_{l=0}^{a+c-h}\binom{a}{l}\binom{b+c+l}{k}\binom{c}{h-a+l}.
\end{align} It can be checked directly that $\alpha_c=\beta_c$ and $\beta_a=\gamma_a$. Since $\binom{c}{l}=0$ for any $l\geq c+1$ and $l\leq-1$, we can unify $\alpha_h,\beta_h,\gamma_h$ by using $\theta_h$ as follows:
\begin{align}
	\theta_h =\sum\limits_{l=0}^h\binom{a}{h-l}\binom{a+b+c-h+l}{k}\binom{c}{l},\ 0\leq h\leq a+c.
\end{align} Denote the two sequences $\mfu=\{u_r\}$ and $\mfv=\{v_r\}$ respectively by 
\begin{align}
	u_r=\binom{a}{r}\binom{a+b+c-r}{k},\ v_r=\binom{c}{r},\ r\in \mbZ_{\geq 0}.
\end{align} By \Cref{lem: binomial1} and \Cref{lem: binomial2}, $\mfu=\{u_r\}$ and $\mfv=\{v_r\}$ are log-concave and have no internal zeros. Hence, based on \Cref{prop: LU}, their convolution $\mfu  * \mfv$ is log-concave and unimodal. Note that \begin{align}
	\theta_h =\sum\limits_{l=0}^hu_{h-l}v_l,\end{align} which implies that $\{\theta_h\}$ is a continuous subsequence of $\mfu * \mfv$. Then, we have $\theta_h^2\geq \theta_{h-1}\theta_{h+1}$ and $\{\theta_h\}$ is unimodal.

\item If $k=b+c+p$ with $1\leq p\leq a$, the associated terms are 
\begin{align}
	\sum\limits_{i=p}^a\sum\limits_{j=0}^{c}\binom{a}{i}\binom{b+c+i}{b+c+p}\binom{c}{j}x_1^{b+c+p}x_2^{a-i+j}x_3^{a-p}.
\end{align} However, the arguments of the log-concavity and unimodality are totally similar to those for the case  $0\leq k\leq b+c$. The only differences are the initial value of $i$ and the value of $k$. Here, we omit the repeating proof but give a special example. When $k=a+b+c$, that is $p=a$, the associated terms are
\begin{align}
	\sum\limits_{j=0}^{c}\binom{c}{j}x_1^{a+b+c}x_2^{j}.
\end{align} This directly implies the log-concavity and unimodality in the direction $x_2$.
\end{enumerate}

In conclusion, all the cluster monomials of the straightforward type cluster algebras of type $A_3$ are log-concave and unimodal.
\subsubsection{Cyclic type} Note that this type corresponds to \Cref{table (3)}. Without loss of generality, we only need to prove that 
$(x_2+x_3)^a(x_1+x_2+x_3)^b$ is log-concave and unimodal for any $a,b\in \mbN$. By similar statements, we can prove that $(x_1+x_2)^a(x_1+x_2+x_3)^b$ and $(x_1+x_3)^a(x_1+x_2+x_3)^b$ are both log-concave and unimodal. Thus, it implies that the cluster monomials of the cyclic type cluster algebra of type $A_3$ are log-concave and unimodal.

Consider the expansion formula of $(x_2+x_3)^a(x_1+x_2+x_3)^b$ as follows.
\begin{align}
& (x_2+x_3)^a(x_1+x_2+x_3)^b\\  =&(x_2+x_3)^a[x_1+(x_2+x_3)]^b \\ =&\sum\limits_{i=0}^b\binom{b}{i}x_1^{b-i}(x_2+x_3)^{a+i} \\ =& 
\sum\limits_{i=0}^b\sum\limits_{j=0}^{a+i}\binom{b}{i}\binom{a+i}{j}x_1^{b-i}x_2^{j}x_3^{a+i-j}.
\end{align} If we fix the exponents of $x_1$ and $x_3$, the exponent of $x_2$ is fixed automatically. Similarly, if we fix the exponents of $x_2$ and $x_3$ (or $x_1$ and $x_2$), the exponent of $x_1$ (or $x_3$) is fixed automatically. Hence, there are no three consecutive terms associated with one variable and it is direct that $(x_2+x_3)^a(x_1+x_2+x_3)^b$ is log-concave and unimodal.

In conclusion, all the cluster monomials of the cyclic type cluster algebras of type $A_3$ are log-concave and unimodal.
\begin{remark}
	In \cite[Theorem 7.20]{GHKK18}, it was proved that all the distinct cluster monomials are linearly independent over $\mbZ$. When the cluster algebra is of finite type, the \emph{theta functions} defined in \cite{GHKK18} are strictly cluster monomials, see also \cite{Nak23}. They encode deep positivity and mirror symmetry properties. Based on \Cref{thm: main result}, it is direct that all the theta functions of type $A_3$ are log-concave and unimodal.
\end{remark}

\subsection{A refined conjecture} In this subsection, based on the results as above, we give a refined conjecture of \Cref{An conj}.

	Note that any two cluster algebras of finite type  (e.g. type $A_n$) are \emph{strongly isomorphic} by taking distinct initial seeds, see \cite{FZ02}. More precisely, it means that there
exists a $\mbZ$-algebra isomorphism $\mcF \rightarrow \mcF^{\prime}$,  sending some seed in $\mathbf{\Sigma}$ into a seed in $\mathbf{\Sigma}^{\prime}$, thus inducing
a bijection $\mathbf{\Sigma} \rightarrow \mathbf{\Sigma}^{\prime}$ of seeds and an algebra isomorphism. We have known that the log-concavity and unimodality of cluster monomials of type $A_3$ are independent of the choice of the initial seed. Hence, it is natural to consider the same property for the higher rank, even though it is still mysterious. 
\begin{remark}\label{rmk: higher rank}
	It is worth emphasizing that the arguments developed for type $A_3$ cannot be extended directly to type $A_4$ or higher. In higher rank, the combinatorial and algebraic structures involved become substantially more complicated, and most technical methods specific to the $A_3$ case are no longer available. To some extent, this also highlights both the importance and the intrinsic uniqueness of this technique in type $A_3$.
\end{remark}
As a consequence of \Cref{thm: main result} and \Cref{rmk: higher rank}, we refine and extend  \Cref{An conj} as follows.
\begin{conjecture}\label{general conj} We conjecture that the following two statements hold.
	\begin{enumerate}
	\item The cluster monomials of type $A_n$ with $n\geq 4$ are log-concave and unimodal.
		\item The strongly isomorphism of cluster algebras keeps the log-concavity and unimodality of cluster monomials of type $A_n$ with $n\geq 4$. Namely, the log-concavity and the unimodality of cluster monomials are independent of the choice of the initial seed.
	\end{enumerate}
\end{conjecture}
\section*{Acknowledgment}
We would like to sincerely thank Zhe Sun for his thoughtful guidance and valuable suggestions. We are grateful to Guanhua Huang and Yilin Wu for many helpful discussions. We also wish to thank Peigen Cao, Xiaowu Chen, Tomoki Nakanishi and Yu Ye for their help and support. This work is supported by National Natural Science Foundation of China (Grant No. 124B2003) and China Scholarship Council (Grant No. 202406340022).


\begin{thebibliography}{99}
	%=========================================================
	\newcommand{\au}[1]{\textrm{#1},}
	\newcommand{\ti}[1]{\textrm{#1},}
	\newcommand{\jo}[1]{\textit{#1}}
	\newcommand{\vo}[1]{\textbf{#1}}
	\newcommand{\yr}[1]{(#1)}
	\newcommand{\pp}[2]{#1--#2.}
	\newcommand{\arxiv}[1]{\href{http://arxiv.org/abs/#1}{arXiv:#1}}
%=========================================================
\bibitem[Bre89]{Bre89}
	\au{F. Brenti}
	\ti{Unimodal, log-concave, and P\'{o}lya frequency sequences in combinatorics} \jo{Mem. Amer. Math. Soc.}, \vo{81}(413), \yr{1989}.
%=========================================================
\bibitem[BH20]{BH20}
\au{P. Br\"and\'en, J. Huh}
\ti{Lorentzian polynomials} \jo{Ann. of Math.}, \vo{192} \yr{2020}, no. 3, \pp{821}{891}
%=========================================================
\bibitem[CHS26]{CHS26}
\au{Z. Chen, G. Huang, Z. Sun}
\ti{Log-concavity of cluster algebras of type $A_n$} 
\jo{Nagoya Math. J.}, \vo{261: e28}
\yr{2026}, \pp{1}{19}
%================================================================

%\bibitem[FG06]{FG06}
%\au{V. Fock, A. Goncharov} \ti{Moduli spaces of local systems and
 % higher Teichm{\"u}ller theory} Publications Math{\'e}matiques de l'Institut
 % des Hautes {\'E}tudes Scientifiques \textbf{103} \yr{2006}, no.~1, \pp{1}{211}

\bibitem[FG09]{FG09} \au{V. Fock, A. Goncharov}
\ti{Cluster ensembles, quantization and the dilogarithm} \jo{Ann. Sci. \'Ec. Norm. Sup\'er.}  \textbf{42}(4), (2009), no. 6, \pp{865}{930}



\bibitem[FST08]{FST08}
	\au{S. Fomin, M. Shapiro, D. Thurston}
	\ti{Cluster algebras and triangulated surfaces. I. Cluster complexes} \jo{Acta Math.}, \vo{201}(1), \yr{2008}, \pp{83}{146}
%=========================================================	
\bibitem[FZ02]{FZ02}
\au{S. Fomin, A. Zelevinsky}
\ti{Cluster Algebra I: Foundations}
\jo{J. Amer. Math. Soc.}
\vo{15} \yr{2002}, \pp{497}{529}
%================================================================
\bibitem[FZ03]{FZ03}
\au{S. Fomin, A. Zelevinsky}
\ti{Cluster Algebra II: Finite type classification}
\jo{Invent. Math.}
\vo{154} \yr{2003}, \pp{112}{164}
%================================================================
%================================================================
\bibitem[FZ04]{FZ04}
\au{S. Fomin, A. Zelevinsky}
\ti{Cluster algebras: Notes for the CDM-03 conference}
\jo{in: CDM 2003: Current Developments in Mathematics, International Press} \yr{2004}.
%================================================================
\bibitem[FZ07]{FZ07}
\au{S. Fomin, A. Zelevinsky}
\ti{Cluster Algebra IV: Coefficients}
\jo{Comp. Math.}
\vo{143}, \yr{2007}, \pp{63}{121}
%================================================================
%\bibitem[GHK15]{GHK15}
%\au{M. Gross, P. Hacking, S. Keel}
%\ti{Birational geometry of cluster algebras}
%\jo{Algebr. Geom.}
%\vo{2}, \yr{2015}, no. 2, \pp{137}{175}


\bibitem[GHKK18]{GHKK18}
\au{M. Gross, P. Hacking, S. Keel, M. Kontsevich}
\ti{Canonical bases for cluster algebras}
\jo{J. Amer. Math. Soc}
\vo{31}, \yr{2018}, \pp{497}{608}
%================================================================
%\bibitem[Gyo21]{Gyo21}  Y. Gyoda, Relation between $f$-vectors and $d$-vectors in cluster algebras of finite type or rank 2, {\it Ann. Comb.} {\bf 25} (2021), 573-594.
%\bibitem[GPY]{GPY}
%	\au{K. Gedeon, N. Proudfoot , B. Young}
%	\ti{Kazhdan-Lusztig polynomials of matroids: a survey of results and conjectures} \yr{2016}, \arxiv{1611.07474}.
%===================
%\bibitem[GS]{GS}
%\au{M. Gross, B. Siebert}
%\ti{From affine geometry to complex geometry}
%\jo{Ann. of Math.}
%\vo{174}, \yr{2011}, no. 3, \pp{1301}{1428}
%================================================================
\bibitem[HLP52]{HLP52}
\au{G. H. Hardy, J. E. Littlewood, G. P\'{o}lya}
\ti{Inequalities, 2nd ed.}
\jo{Cambridge Univ. 
Press, Cambridge} \yr{1952}.
%===================	
 \bibitem[HMMS22]{HMMS22}
	\au{J. Huh, J. Matherne, K. M{\'e}sz{\'a}ros, A. St Dizier}
	\ti{Logarithmic concavity of Schur and related polynomials}
	\jo{Trans. Amer. Math. Soc.} \vo{375}, \yr{2022}, no. 6,  \pp{4411}{4427}
 
 %===================	
  %===================	
  \bibitem[IIKKN13a]{IIKKN13a}
	\au{R. Inoue, O. Iyama, B. Keller, A. Kuniba, T. Nakanishi}
	\ti{Periodicities of $T$ and $Y$-systems I: Types $B_r$}
	\jo{Publ. RIMS.} \vo{49}, \yr{2013},  \pp{1}{42}
 %===================	
  \bibitem[IIKKN13b]{IIKKN13b}
	\au{R. Inoue, O. Iyama, B. Keller, A. Kuniba, T. Nakanishi}
	\ti{Periodicities of $T$ and $Y$-systems II: Types $C_r$, $F_4$, and $G_2$}
	\jo{Publ. RIMS.} \vo{49}, \yr{2013},  \pp{43}{85}
 %===================	
 \bibitem[IKFP13]{IKFP13}
	\au{G. Cerulli Irelli, B. Keller, D. Labardini-Fragoso, P-G.
Plamondon}
	\ti{Linear independence of cluster monomials for skew-symmetric cluster algebras}
	\jo{Compos. Math.} \vo{149}(10), \yr{2013}, no. 6,  \pp{1753}{1764}

%===================
%===================
\bibitem[IY16]{IY16}
\au{I. C. H. Ip, M. Yamazaki}
\ti{Quantum dilogarithm identities at root of unity}
\jo{Int. Math.
Res. Not. IMRN}, \yr{2016}, no. 3, \pp{669}{695}

%===================
\bibitem[Wil14]{Wil14}
\au{L. K. Williams}
\ti{Cluster algebras: an introduction}
\jo{Bull. Amer. Math. Soc.}
\vo{51}, \yr{2014}, \pp{1}{26}

%===================
\bibitem[Man17]{Man17}
\au{T. Mandel} \ti{Theta bases are atomic}
\jo{Compos. Math.} \vo{153}, \yr{2017}, \pp{1217}{1219}
%\bibitem[Man17]{Man17}
%\au{T. Mandel} \ti{Theta bases are atomic} \jo{Compos. Math.} \vo{153}, \yr{2017}, \pp{1217}{"1219}
%\bibitem[MQ23]{MQ23}
%T. Mandel, F. Qin, \emph{Bracelet bases are theta bases}, arXiv preprint, arXiv:2301.11101.
	\bibitem[Nak23]{Nak23}
\au{T. Nakanishi}
\ti{Cluster algebras and scattering diagrams}
\jo{MSJ Mem.} \vo{41} \yr{2023}, 279 pp; ISBN: 978-4-86497-105-8.
	%===================
	\bibitem[Nak24]{Nak24} 
	\au{T. Nakanishi} \ti{Dilogarithm identities in cluster scattering diagrams}, \jo{Nagoya Math. J.} \vo{253} (2024), \pp{1}{22}
%\bibitem[Nak24]{Nak24}
%\au{T. Nakanishi} \ti{Dilogarithm identities in cluster scattering diagrams}, \jo{Nagoya Math. J.} \vo{253} (2024), \pp{1}{"22}
	%===================
	%\bibitem[N2]{N2}
%\au{T. Nakanishi}
%\ti{Synchronicity phenomenon in cluster patterns}
%\jo{J. London Math. Soc.}
%\vo{103}, \yr{2021} \pp{1120}{1152}
%===================
	\bibitem[Oko03]{Oko03}
	\au{A. Okounkov}
	\ti{Why would multiplicities be log-concave?}
	\jo{The orbit method in geometry and physics
		(Marseille, 2000), Progr. Math.} \vo{213}, Birkh$\ddot{a}$user Boston, Boston, MA, \yr{2003}, no. 1,  \pp{329}{347}

%\bibitem[Pen87]{Pen87}
%R. Penner, \emph{The decorated {T}eichm\"uller space of punctured surfaces},
%  Comm. Math. Phys. \textbf{113} (1987), no.~2, 299--339.
%===================
\bibitem[OLBC10]{OLBC10}
\au{F. Olver, D. Lozier, R. Boisvert,C. Clark}
\ti{NIST Handbook of Mathematical Functions}
\jo{Cambridge Univ. Press}, \yr{2010}.
%===================
  
%===================	
\bibitem[Sch08]{Sch08}
\au{R. Schiffler}
\ti{A cluster expansion formula ($A_n$ case)}
\jo{Electron. J. Comb.}
\vo{15}(1): Research paper 64, 9, \yr{2008}. 

\bibitem[Shen14]{Shen14}
\au{L. Shen}
\ti{Stasheff polytopes and the coordinate ring of the cluster $X$-variety of type $A_n$}
\jo{Selecta Mathematica (New Series)} \vo{20} \yr{2014}, no. 3,  \pp{929}{959}


%===================



\bibitem[ST09]{ST09}
\au{R. Schiffler, H. Thomas}
\ti{On cluster algebras arising from unpunctured surfaces}
\jo{Int. Math. Res. Not.}
17, \yr{2009}, \pp{3160}{3189}
%===================
\bibitem[Sta89]{Sta89}
\au{R. P. Stanley}
\ti{Log-concave and unimodal sequences in algebra, combinatorics, and geometry}
\jo{Ann. New York Acad. Sci} \vo{576} \yr{1989}, no. 1,  \pp{500}{535}
%===================	
%===================
\bibitem[Sze75]{Sze75}
\au{G. Szeg{\"o}}
\ti{Orthogonal Polynomials, 4th ed.}
\jo{Amer. Math. Soc. Colloq. Publ.}, \vo{23}, Providence, RI, \yr{1975}.%===================

%\bibitem[Thu79]{Thu79}
%W. P. Thurston, \emph{The geometry and topology of three-manifolds}, Princeton, NJ: Princeton University, 1979.

%=========================================================
	
\end{thebibliography}
\end{document}